\documentclass[10pt]{amsart}

\usepackage[usenames,dvipsnames]{color}
\usepackage{amsmath,amsthm}
\usepackage{amssymb,latexsym}

\makeatletter
\@addtoreset{equation}{section}
\makeatother

\newtheorem{theorem}{Theorem}[section]
\newtheorem{proposition}[theorem]{Proposition}
\newtheorem{lemma}[theorem]{Lemma}
\newtheorem{corollary}[theorem]{Corollary}

\theoremstyle{definition}

\newtheorem{remark}[theorem]{Remark}
\newtheorem{example}[theorem]{Example}

\newcommand{\wbar}[1]{\overline{#1}}
\newcommand{\what}[1]{\widehat{#1}}

\newcommand{\til}[1]{\tilde{#1}}

\newcommand{\id}{\operatorname{id}}
\newcommand{\spn}{\operatorname{span}}

\newcommand{\fA}{\mathcal{A}}
\newcommand{\fB}{\mathcal{B}}
\newcommand{\fC}{\mathcal{C}}
\newcommand{\fD}{\mathcal{D}}

\newcommand{\fH}{\mathcal{H}}
\newcommand{\fK}{\mathcal{K}}
\newcommand{\fL}{\mathcal{L}}
\newcommand{\fM}{\mathcal{M}}

\newcommand{\fV}{\mathcal{V}}
\newcommand{\fW}{\mathcal{W}}
\newcommand{\fX}{\mathcal{X}}

\newcommand{\Cee}{\mathbb{C}}
\newcommand{\Ree}{\mathbb{R}}
\newcommand{\Tee}{\mathbb{T}}
\newcommand{\Zee}{\mathbb{Z}}
\newcommand{\En}{\mathbb{N}}

\newcommand{\alp}{\alpha}
\newcommand{\del}{\delta}
\newcommand{\Del}{\Delta}
\newcommand{\eps}{\varepsilon}
\newcommand{\gam}{\gamma}
\newcommand{\Gam}{\Gamma}
\newcommand{\lam}{\lambda}
\newcommand{\Lam}{\Lambda}
\newcommand{\ome}{\omega}
\newcommand{\Ome}{\Omega}
\newcommand{\sig}{\sigma}
\newcommand{\Sig}{\Sigma}

\newcommand{\norm}[1]{\left\Vert#1\right\Vert}

\newcommand{\cent}{\mathrm{Z}}
\newcommand{\col}{\mathrm{C}}
\newcommand{\fal}{\mathrm{A}}
\newcommand{\pfal}{\mathrm{A}^p}
\newcommand{\ideal}{\mathrm{I}}
\newcommand{\mat}{\mathrm{M}}
\newcommand{\row}{\mathrm{R}}
\newcommand{\smat}{\mathrm{S}}
\newcommand{\trig}{\mathrm{Trig}}
\newcommand{\su}{\mathrm{SU}}
\newcommand{\un}{\mathrm{U}}

\newcommand{\ext}{\mathrm{ext}}

\newcommand{\trace}{\mathrm{Tr}}

\begin{document}

\title[$p$-Fourier algebras]
{$p$-Fourier algebras on compact groups}

\author{Hun Hee Lee, Ebrahim Samei and Nico Spronk}

\begin{abstract}
Let $G$ be a compact group.  For $1\leq p\leq\infty$ we introduce a class of Banach function algebras
$\pfal(G)$ on $G$ which are the Fourier algebras in the case $p=1$, and for $p=2$ are certain algebras
discovered in \cite{forrestss1}.  In the case $p\not=2$ we find that $\pfal(G)\cong\pfal(H)$ if and only if $G$ 
and $H$ are isomorphic compact groups.  These algebras admit natural operator space structures,
and also weighted versions, which we call $p$-Beurling-Fourier algebras.
We study various amenability and operator amenability properties, Arens regularity and
representability as operator algebras.  For a connected Lie $G$ and $p>1$, our techniques
of estimation of when certain $p$-Beurling-Fourier algebras are operator algebras rely
more on the fine structure of $G$, than in the case $p=1$.  We also study restrictions
to subgroups.  In the case that $G=\su(2)$, restrict to a torus and obtain some exotic algebras
of Laurent series.  We study amenability properties of these new algebras, as well.
\end{abstract}

\maketitle

\footnote{{\it Date}: \today.


The first named author would like to thank the Basic Science Research Program through the National Research Foundation of Korea (NRF), grant NRF-2012R1A1A2005963.
The second named author would like to thank NSERC Grant 366066-2009.
The third named author would
like thank NSERC Grant 312515-2010, and the Korean Brain Pool Program 2014 supported by MSIP(Ministry of Science, ICT \& Future 
Planning).}


\subsection{Introduction and plan}
For a compact group $G$, the Fourier algebra $\fal(G)$
goes back to Kre\u{\i}n (see \cite[\S 34]{hewittrII}, for example).  For a continuous function $u$ on $G$
\[
u\in\fal(G)\quad\text{if and only if}\quad\norm{u}_\fal=\sum_{\pi\in\what{G}}d_\pi
\norm{\hat{u}(\pi)}_{\smat^1_{d_\pi}}<\infty.
\]
Here, each $\hat{u}(\pi)$ is the matricial Fourier coefficient, and each $\norm{\cdot}_{\smat^1_d}$ denotes the 
trace norm.  This is a special version of the Fourier algebra as subsequently defined by Stinespring 
(\cite{stinespring}) and Eymard (\cite{eymard}).  
In his investigation of non-amenability of $\fal(\mathrm{SO}(3))$, Johnson (\cite{johnson}) created a new
algebra $\fal_\gam(G)$ by averaging over the diagonal subgroup. More precisely, he considered the map
from the projective tensor product of Fourier algebras into continuous functions on a homogeneous space
modulo the diagonal subgroup $\Del = \{(s,s): s\in G\}$:
\begin{gather*}
\Gam: \fal(G) \otimes_\gam \fal(G) \to \fC(G\times G/\Del)\cong\fC(G) \\
\Gam(u\otimes v)((s,t)\Del)= \int_G u(sr)v(tr)\,dr=u\ast\check{v}(st^{-1})
\end{gather*}
where $\check{v}(r)=v(r^{-1})$.  
Letting $\fal_\gam(G)=\Gam(\fal(G)\otimes_\gam\fal(G))$, we have
\[
u\in\fal_\gam(G)\quad\text{if and only if}\quad\norm{u}_{\fal_\gam}=\sum_{\pi\in\what{G}}d_\pi^2
\norm{\hat{u}(\pi)}_{\smat^1_{d_\pi}}<\infty
\]
and $\norm{\cdot}_\gam$ is exactly that norm which makes $\Gam$ a quotient map.
In \cite{forrestss1}. Forrest and two of the present authors followed the same philosophy to consider the operator space version of Johnson's algebra, $\fal_\Del(G) = \Gam(\fal(G\times G))$, where we recall
that $\fal(G\times G)=\fal(G)\hat{\otimes}\fal(G)$ (operator projective tensor product).
A characterization of the algebra in terms of its natural quotient norm is given by 
\[
u\in\fal_\Del(G)\quad\text{if and only if}\quad\norm{u}_{\fal_\Del}=\sum_{\pi\in\what{G}}d_\pi^{3/2}
\norm{\hat{u}(\pi)}_{\smat^2_{d_\pi}}<\infty
\]
where each $\norm{\cdot}_{\smat^2_d}$ denotes the Hilbert-Schmidt norm. This result was mysterious to the authors at the time, and there were questions about the underlying meaning of the formula.  We  remark that the space $\fal_\Del(G)$ seems to have appeared before.  It arises as a host space for sufficiently differentiable functions on compact Lie groups; see \cite{sugiura} or \cite[12.2.2]{faraut}, for example.

We provide an answer to this question of understanding $\fal_\gam(G)$ and $\fal_\Del(G)$, 
in a uniform context with $\fal(G)$,
by developing a family of algebras $\pfal(G, \ome)$, $1\leq p\leq\infty$, where $\ome : \what{G} \to (0,\infty)$ is a weight in the sense of \cite{lees} or \cite{ludwigst}.
We recover the Fourier algebra $\fal(G)$ when $p=1$ and $\ome\equiv 1$, the algebra $\fal_\gam(G)$ when $p=1$ and $\ome(\pi) = d_\pi$, and the algebra $\fal_\Del(G)$ when $p=2$ and $\ome\equiv 1$. We may simply refer to each of these as the {\it $p$-Beurling--Fourier algebra} of $G$ and the {\it $p$-Fourier algebra} of $G$ for the case of $\ome \equiv 1$, which we simply write $\pfal(G)$.  We immediately warn the reader that the algebras $\pfal(G)$ are distinct from, and not to be confused with, the Fig\`{a}-Talamanca--Herz algebras $\fal_p(G)$, $1<p<\infty$.  For example, we observe for virtually abelian groups that $\pfal(G)=\fal(G)$ isomorphically (Theorem \ref{prop:openabel}), while this does not hold for $\fal_p(G)$, even for general compact abelian groups.  In particular, our class of algebras is new only for sufficiently non-abelian compact groups.

Having constructed this class of algebras $\pfal(G,\ome)$, we embark on a systematic program of understanding how its members behave as Banach algebras. In general, the algebras tend to reflect the structure of $G$ in a similar way as $\fal(G)$ did. For example, they still determine the underlying compact group $G$ for most $p$ and for some natural weights $\ome$ (Section \ref{ssec:isomorphism}). The Gelfand spectrum of $\pfal(G,\ome)$ still identical to $G$ for sufficiently slowly growing weights, and for dimension weights these algebras are regular (Section \ref{ssec:spectrum}).  As is noted in Section \ref{ssec:amenprop},
the algebras $\pfal(G)$ enjoys exactly the same amenability properties of the Fourier algebras.

However, the algebras $\pfal(G,\ome)$ show significant differences from Fourier algebras, in particualr when we look at reults in the context of operator spaces.  Each $\pfal(G, \ome)$ comes equipped with a natural operator space structure, generalizing the usual ``predual" operator space structure of $\fal(G)$, via the complex interpolation.  For $p=2$ other operator space structures are available, which arise naturally in some contexts.  As has been amply demonstrated in the case of Fourier algebras the use of operator space structure makes analysis on non-abelian groups more tractable: compare the results of \cite{johnson} and \cite{forrestr} with those of  \cite{ruan} and \cite{spronk,samei}. We examine operator weak amenability and operator amenability in Sections \ref{ssec:opamenprop}, where we can observe that the corresponding properties usually become worse than the the case of Fourier algebra. More intriguing difference is that the functoriality of restriction to subgroups. It turns out that the functoriality holds in a very restricted case, namely, the case where the subgroup is more or less a direct summand of the supergroup (Section \ref{ssec:directproduct}). This limited functoriality already gives us additional results on operator weak amenability and operator amenability of $\pfal(G,\ome)$.  On the other hand the restriction procedure to subgroups is not functorial in general. We demonstrate its failure focusing on the case of $2\times 2$ special unitary group, $\su(2)$ in Section \ref{ssec:tori}, since its representation theory allows more direct computations. Interestingly, we gain an exotic class of Banach algebras of Laurent series as a result of restriction to the maximal torus. We characterize various amenabilities of these new algebras, which shows unexpected coincidence of amenabilities in Banach space category and operator space category.

Finally in Section \ref{ssec:arens} we conduct a study of Arens regularity of algebras $\pfal(G,\ome)$ and in Section \ref{ssec:opalg} we study, in the case of connected Lie groups, some conditions which allow $\pfal(G,\ome)$ to be represented as operator algebras. The latter property in particular shows some similarity and difference at the same time from the case of $p=1$ in \cite{ghandeharilss}. It is similar in the sense that we still can make the algebras $\pfal(G,\ome)$ to be represented as operator algebras for the weights growing fast enough. It is also quite different from the case of \cite{ghandeharilss}, where our estimates for the growth rate of the corresponding weight relied only on the dimension $d(G)$ of the connected Lie group $G$, but the present work shows the connection to a finer structure information of $G$ such as semisimple rank $s(G)$ of $G$ and the dimension $z(G)$ of the center of $G$.

Since this paper features a wide variety of different results, many of them quantitative, we include two tables in the last section which summarize most of the results of this paper in the context of known results.  We also point out some questions which arise naturally form the present investigation.

\subsection{Fourier series on compact groups}
For a compact group $G$, we let $\what{G}$ denote its dual object, the set of unitary equivalence
classes of all irreducible unitary representations.  By standard abuse of notation, we shall
identify elements of $\what{G}$ with representatives of the equivalence classes.
Each $\pi$ in $\what{G}$ acts on a space $\fH_\pi$ of dimension $d_\pi$.
Let $\mat_d$ denote the space of $d\times d$ complex matrices.
Given $\pi$ in $\what{G}$ we let $\trig_\pi=\{\trace(A\pi(\cdot)):A\in\mat_{d_\pi}\}$
and $\trig(G)=\bigoplus_{\pi\in\what{G}}\trig_\pi$, where the sum
may be regarded as an internal direct sum in the space of continuous functions
on $G$.  Given $u=\sum_{\pi\in\what{G}}d_\pi\trace(A_\pi\pi(\cdot))$ (finite sum) in
$\trig(G)$, Schur's orthogonality formula provides that $A_\pi=\int_G u(s)\pi(s^{-1})\,ds$
(normalized Haar integration).  We denote this matricial coefficient $A_\pi=\hat{u}(\pi)$.
The linear space $\trig(G)$ admits algebraic dual space 
$\trig(G)^\dagger=\prod_{\pi\in\what{G}}\mat_{d_\pi}$ via the duality
\begin{equation}\label{eq:dualpairing}
\langle u,(T_\pi)_{\pi\in\what{G}}\rangle =\sum_{\pi\in\what{G}}d_\pi\trace(\hat{u}(\pi)T_\pi).
\end{equation}

This duality is constructed specifically to recognize the functionals of evaluation at points:
for $s$ in $G$, let $\lam(s)=(\pi(s))_{\pi\in\what{G}}$, and we have $u(s)=\langle u,\lam(s)\rangle$
for each $u$ in $\trig(G)$.  Furthermore, a consequence of Schur's lemma is that
$\spn\lam(G)$ is weak$^\dagger$ dense in $\trig(G)^\dagger$, i.e.\ dense with respect to
the initial topology $\sig(\trig(G)^\dagger,\trig(G))$ with respect to the 
dual pairing (\ref{eq:dualpairing}).  We let 
$m:\trig(G)\otimes\trig(G)\cong\trig(G\times G)\to\trig(G)$ denote
the pointwise product, and define the {\it coproduct} $M:\trig(G)^\dagger\to\trig(G\times G)^\dagger$  
to be the adjoint of $m$.  Hence $M$ is defined by the relation
\[
M\lam(s)=\lam(s)\otimes\lam(s)\text{ for }s\text{ in }G
\] 
which is the familiar co-commuative coproduct from the theory of (compact) quantum groups.
We shall require a form of this which is more suitable for certain norm computations, however.
If $\sig$ is equivalent to a subrepresentation of $\pi\otimes\pi'$ -- we write
$\sig\subset\pi\otimes\pi'$ in this case --  we let
$U_{\sig,\pi\otimes\pi'}^{(i)}:\fH_\sig\to\fH_\pi\otimes^2\fH_{\pi'}$, $i=1,\dots,m(\sig,\pi\otimes\pi')$,
denote a maximal family of  isometric intertwiners with pairwise disjoint ranges.
Then for suitably sized matrices $A_\pi$ and $A_{\pi'}$ we see that
\[
m(\trace(A_\pi\pi(\cdot))\otimes\trace(A_{\pi'}\pi'(\cdot)))=\sum_{\sig\subset\pi\otimes\pi'}\!\!
\sum_{i=1}^{m(\sig,\pi\otimes\pi')}\trace(A_\pi\otimes A_{\pi'}U_{\sig,\pi\otimes\pi'}^{(i)}\sig(\cdot)
U_{\sig,\pi\otimes\pi'}^{(i)*}).
\]
We note that there is a natural linear isomorphism $\trig(G\times G)=\trig(G)\otimes\trig(G)$,
and that $\trig(G\times G)^\dagger=\prod_{\pi,\pi'\in\what{G}}\mat_{d_\pi}\otimes\mat_{d_{\pi'}}$.
Hence by the weak$^\dagger$ density of $\spn\lam(G)$ in $\trig(G)^\dagger$, and the
weak$^\dagger$ density of $\bigoplus_{\pi,\pi'\in\what{G}}\mat_{d_\pi}\otimes\mat_{d_{\pi'}}$
in $\trig(G\times G)^\dagger$, we have that
\begin{equation}\label{eq:coprod}
M W_{\pi,\pi'}=\bigoplus_{\sig\subset\pi\otimes\pi'}U_{\sig,\pi\otimes\pi'}^*W_\sig
U_{\sig,\pi\otimes\pi'}
\end{equation}
where the sum now counts multiplicity, i.e.\ we absorb indices $i=1,\dots,m(\sig,\pi\otimes\pi')$
so as to simplify the notation in our future computations.  It is well-known, see, for example,
\cite{stinespring} -- that $M(\fal(G)^*)\subset\fal(G)^*\bar{\otimes}\fal(G)^*$ where
$\bar{\otimes}$ denotes the normal spatial, or von Neumann, tensor product.

We note that for any norm $\norm{\cdot}_\fA$ on $\trig(G)$, the completion $\fA(G)$
of $\trig(G)$ with respect to that norm may be understood to have its continuous dual space
$\fA(G)^*$ as a subspace of $\trig(G)^\dagger$.  For example, the Fourier algebra, mentioned
earlier, satisfies
\[
\fal(G)\cong\ell^1\text{-}\bigoplus_{\pi\in\what{G}}d_\pi\smat^1_{d_\pi}\quad\text{and}\quad
\fal(G)^*\cong\ell^\infty\text{-}\bigoplus_{\pi\in\what{G}}\smat^\infty_{d_\pi}
\]
i.e.\ Banach space direct sum of weighted trace-class matrix spaces and direct product
of operators on $d_\pi$-dimensional Hilbert spaces, respectively.

\subsection{Operator spaces}
Our standard references for the theory of operator spaces are \cite{effrosrB} and \cite{pisier1}.
An operator space is a complex vector space $\fV$ admitting a family of norms
$\norm{\cdot}_{\mat_n(\fV)}$, one on each of the spaces of $n\times n$ matrices with entries
in $\fV$, satisfying certain compatibility conditions of Ruan (which we shall not explicitly use), 
and for which each $\mat_n(\fV)$ is complete.  For us, the space $\smat_d^\infty$ 
will always have the canonical operator space structure it admits by virtue
of being a C*-algebra.  The space $\smat^1_d$ will come by its operator space structure
as the dual space of $\smat_d^\infty$.   We shall make use of direct products and
direct sums (the latter of which is described nicely in \cite{pisier1}), denoted
$\oplus^\infty$ or $\ell^\infty\text{-}\bigoplus$, and 
$\oplus^1$ or $\ell^1\text{-}\bigoplus$, respectively; as well as operator projective,
operator injective (or spatial), and Haagerup tensor products, denoted $\hat{\otimes},\check{\otimes}$ and
$\otimes^h$, respectively.  We shall also often use the normal spatial tensor
product $\bar{\otimes}$ of operator dual spaces.  In particular, there are distributive laws over
infinite families of operator space
of operator direct sums and the operator projective tensor product, and 
over operator dual spaces of operator direct products
and the normal spatial tensor product.

The spaces $\smat_d^p$, with Schatten $p$-norms,
will have their operator space structures realized through Pisier's complex interpolation
theory (\cite{pisier}):  $\smat_d^p=[\smat_d^\infty,\smat^1_d]_{1/p}$.  Hence the Hilbert-Schmidt
space $\smat^2_d$ will be understood with the operator Hilbert space structure.  However, we
will also have occasions to consider this space as a column Hilbert space, $\smat^2_{d,C}$,
or a row Hilbert space, $\smat^2_{d,R}$.  We shall do the same with all of the $d$-dimensional
spaces $\ell^p_d$.  In this case  $\ell^\infty_d$, being
a commutative C*-algebra, is a minimal operator space, while $\ell^1_d$, being its dual, is maximal.
In the case $p=2$,
we let $\col_d=\ell^2_{d,C}$ and $\row_d=\ell^2_{d,R}$ denote the $d$-dimensional
column and row spaces.  We shall use the bilinear identifications $\col_d^*\cong\row_d$
and $\row_d^*\cong\col_d$, given by $(\xi,\eta)\mapsto \eta\xi, \xi\eta$, respectively.  
Thus we obtain completely isometric identifications
\[
\smat^\infty_d=\col_d\otimes^h\row_d\quad\text{and}\quad\smat^1_d=\row_d\otimes^h\col_d.
\]
Then, keeping with standard notation of \cite{pisier1}, we obtain interpolated Hilbertian operator
spaces
\[
\col_d^p=[\col_d,\row_d]_{1/p}\quad\text{and}\quad\row_d^p=[\row_d,\col_d]_{1/p}.
\]
In this notation we have $\col^\infty_d=\col_d$ and $\col^1_d=\row_d$, for example.
This notation proves itself in the remarkable tensorial factorizations
\[
\smat^p_d=\col_d^p\otimes^h\row^p_d
\] 
which hold thanks to the stability of Haagerup tensor
product under complex interpolation (\cite{pisier}).

It is well-known that the space of completely bounded maps $\fC\fB(\col_d,\row_d)$ is isometrically
isomorphic to $\smat^2_d$.  Hence the identity $\id:\col_d\to\row_d$ has completely bounded
norm $d^{1/2}$.
We abbreviate this by writing that $\id:d^{1/2}\col_d\to\row_d$ is a complete contraction.

We realize $p$-direct sums through interpolation:  $\fV\oplus^p\fW=[\fV\oplus^\infty\fW,
\fV\oplus^1\fW]_{1/p}$, completely isometrically.

\section{The algebras $\pfal(G)$}

Let $G$ be a compact group.

\subsection{Definition and elementary properties of $\pfal(G)$}
Let $1\leq p\leq\infty$ and $p'$ denote the conjugate index so $\frac{1}{p}+\frac{1}{p'}=1$.
We consider on $\trig(G)$ the norm
\[
\norm{u}_{\pfal}=\sum_{\pi\in\what{G}}d_\pi^{1+\frac{1}{p'}}\norm{\hat{u}(\pi)}_p.
\]
Let $\pfal(G)$ denote the completion of $\trig(G)$ with respect to this norm, hence
we have isometric identification
\begin{equation}\label{eq:pfal}
\pfal(G)\cong\ell^1\text{-}\bigoplus_{\pi\in\what{G}}d_\pi^{1+\frac{1}{p'}}\smat_{d_\pi}^p.
\end{equation}
We may further put the canonical operator space structures on the individual spaces
$\smat_{d_\pi}^p$ and use operator space direct sum, thus 
making (\ref{eq:pfal}) a completely isometric identification.
We note that for any positive integer $d$, $(d\smat_d^\infty,\smat_d^1)$ forms a
natural compatible couple of operator spaces with interpolated spaces 
$[d\smat_d^\infty,\smat_d^1]_{1/p}=d^{1/p'}\smat_d^p$.  Thus since
the formal identity $d\smat_d^\infty\cong d^{1/2}\col_d\otimes^h d^{1/2}\row_d\hookrightarrow
\row_d\otimes^h\col_d\cong\smat_d^1$ is a complete contraction,
it follows standard operator interpolation theory
we get for $1\leq p\leq q\leq \infty$ completely contractive identity maps $d\smat_d^\infty
\hookrightarrow d^{1/q'}\smat_d^q\hookrightarrow d^{1/p'}\smat_d^p \hookrightarrow \smat_d^1$.
These give rise to completely contractive inclusions
\begin{equation}\label{eq:ccinclusions}
\fal^\infty(G)\subseteq\fal^q(G)\subseteq\pfal(G)\subseteq\fal(G).
\end{equation}
Hence each space is a space of continuous functions on $G$.

Let us see that each $\pfal(G)$ is a completely contractive Banach algebra. 
Hence we are justified in calling it the {\it $p$-Fourier algebra} of $G$.  We first observe
that the dual space is the direct product space
\begin{equation}\label{eq:pfaldual}
\pfal(G)^*\cong\ell^\infty\text{-}\bigoplus_{\pi\in\what{G}}d_\pi^{-\frac{1}{p'}}\smat_{d_\pi}^{p'}.
\end{equation}
Also, since the operator projective tensor product commutes with operator space direct sums
we have that
\begin{equation}\label{eq:pfaldualt}
\pfal(G)\hat{\otimes}\pfal(G)\cong\underset{\pi,\pi'\in\what{G}\times\what{G}}{\ell^1\text{-}\bigoplus}
(d_\pi d_{\pi'})^{1+\frac{1}{p'}}\smat_{d_\pi}^p\hat{\otimes}\smat_{d_{\pi'}}^p
\end{equation}
and hence we obtain
\begin{equation}\label{eq:pfaltpdual}  
(\pfal(G)\hat{\otimes}\pfal(G))^*\cong
\underset{\pi,\pi'\in\what{G}\times\what{G}}{\ell^\infty\text{-}\bigoplus}
(d_\pi d_{\pi'})^{-1/p'}\smat_{d_\pi}^{p'}\check{\otimes}\smat_{d_{\pi'}}^{p'}.
\end{equation}

\begin{theorem}\label{theo:ccBa}
The space $\pfal(G)$ is a completely contractive Banach algebra under pointwise multiplication.
\end{theorem}

\begin{proof}
It suffices to show that the coproduct $M$ of (\ref{eq:coprod}) 
takes $\fal(G)^*$ into $(\pfal(G)\hat{\otimes}\pfal(G))^*$, 
completely contractivley.  Now if $[W_{ij}]\in\mat_n(\pfal(G)^*)$ then we wish to estimate
\begin{equation}\label{eq:normgamw}
\norm{[M W_{ij}]}=\sup_{\pi,\pi'\in\what{G}\times\what{G}}(d_\pi d_{\pi'})^{-1/p'}
\norm{\left[\bigoplus_{\sig\subset\pi\otimes\pi'}U_{\sig,\pi\otimes\pi'}^*W_{ij,\sig}
U_{\sig,\pi\otimes\pi'}\right]}_{\mat_n(\smat_{d_\pi}^{p'}\check{\otimes}\smat_{d_{\pi'}}^{p'})}.
\end{equation}
Let us make three general observations.  

First, if $d$ and $d'$ be two positive integers,
the formal identity $\smat_{dd'}^{p'}\hookrightarrow\smat_d^{p'}\check{\otimes}\smat_{d'}^{p'}$ 
is a complete contraction.  Indeed, the desired map is the adjoint of the map 
$\smat_d^p\hat{\otimes}\smat_{d'}^p\hookrightarrow \smat_{d d'}^p$, which is completely 
contractive since we may recognize it as
\begin{align*}
\smat_d^p\hat{\otimes}\smat_{d'}^p
&\cong (\col_d^p\otimes^h\row_d^p)\hat{\otimes}(\col_{d'}^p\otimes^h\row_{d'}^p) \\
&\hookrightarrow (\col_d^p\hat{\otimes}\col_{d'}^p)\otimes^h(\row_d^p\hat{\otimes}\row_{d'}^p) 
\tag{$\dagger$} \\
&\hookrightarrow (\col_d^p\otimes^h\col_{d'}^p)\otimes^h(\row_d^p\otimes^h\row_{d'}^p) \\
&\cong \col_{dd'}^p\otimes^h\row_{dd'}^p\cong \smat_{dd'}^p
\end{align*}
where the complete contraction $c_1\otimes r_1\otimes c_2\otimes r_2\mapsto
c_1\otimes c_2\otimes r_1\otimes r_2$
at $(\dagger)$ is provided by \cite[Thm.\ 6.1]{effrosr1}. 

Second, if $d_1,\dots,d_m$ are positive integers and
$d=d_1+\dots+d_m$, then  the block-diagonal embedding of the operator space 
$\ell^{p'}\text{-}\bigoplus_{j=1}^m\smat_{d_j}^{p'}$ into $\smat_d^{p'}$ is a complete isometry.
Indeed when $p'=\infty$, this is a completely contractively complemented subspace, whence, by
duality, the same holds when $p'=1$.  The case for general $p$ then follows from the generalized 
Riesz-Thorin Theorem, using the facts that $\ell^{p'}\text{-}\bigoplus_{j=1}^m\smat_{d_j}^{p'}=
[\ell^\infty\text{-}\bigoplus_{j=1}^m\smat_{d_j}^\infty,\ell^1\text{-}\bigoplus_{j=1}^m\smat_{d_j}^1
]_{1/p'}$.  

Third, if $E$ and $F$ are operator spaces then the formal identity $\mat_n(E)\oplus^{p'}\mat_n(F)
\hookrightarrow \mat_n(E\oplus^{p'} F)$ is a contraction.  Indeed, this map is an isometry
if $p'=\infty$, and, thanks to the universal property of direct sums, is a contraction when $p'=1$.
The case for general $p'$ follows from the generalized Riesz-Thorin Theorem, using the facts that
$\mat_n(E)\oplus^{p'}\mat_n(F)=[\mat_n(E)\oplus^\infty\mat_n(F),\mat_n(E)\oplus^1\mat_n(F)
]_{1/p'}$ and $\mat_n(E\oplus^{p'} F)=[\mat_n(E\oplus^\infty F),\mat_n(E\oplus^1 F)]_{1/p'}$.

Using the three observations above, in order, and then a rudimentary H\"{o}lder estimate we obtain
\begin{align*}
&\norm{\left[\bigoplus_{\sig\subset\pi\otimes\pi'}U_{\sig,\pi\otimes\pi'}^*W_\sig
U_{\sig,\pi\otimes\pi'}\right]}_{\mat_n(\smat_{d_\pi}^{p'}\check{\otimes}\smat_{d_{\pi'}}^{p'})} \\
&\qquad \leq 
\norm{\left[\bigoplus_{\sig\subset\pi\otimes\pi'}^*U_{\sig,\pi\otimes\pi'}W_{ij,\sig}
U_{\sig,\pi\otimes\pi'}\right]}_{\mat_n(\smat_{d_\pi d_{\pi'}}^{p'})} \\
&\qquad =\norm{\bigoplus_{\sig\subset\pi\otimes\pi'}[W_{ij,\sig}]}_{\mat_n(\ell^{p'}\text{-}
\bigoplus_{\sig\subset\pi\otimes\pi'}\smat_{d_\sig}^{p'})} \\
&\qquad \leq \left(\sum_{\sig\subset\pi\otimes\pi'}\norm{[W_{ij,\sig}]
}_{\mat_n(\smat_{d_\sig}^{p'}}^{p'})\right)^{1/p'} \\
&\qquad \leq \left(\sum_{\sig\subset\pi\otimes\pi'}d_\sig\cdot\sup_{\tau\subset\pi\otimes\pi'}
\frac{1}{d_\tau}\norm{[W_{ij,\tau}]}_{\mat_n(\smat_{d_\tau}^{p'})}^{p'}\right)^{1/p'} \\
&\qquad = (d_\pi d_{\pi'})^{1/p'}
\sup_{\tau\subset\pi\otimes\pi'}d_\tau^{-1/p'}\norm{[W_{ij,\tau}]}_{\mat_n(\smat_{d_\tau}^{p'})}.
\end{align*}
It follows that the quantity of (\ref{eq:normgamw}) is no greater than
\begin{equation}\label{eq:normgamw1}
\sup_{\tau\in\what{G}}d_\tau^{-1/p'}\norm{[W_{ij,\tau}]}_{\mat_n(\smat_{d_\tau}^{p'})}
=\norm{[W_{ij}]}_{\mat_n(\pfal(G)^*)}.
\end{equation}
Hence $M$ satisfies the desired complete contractivity property.
\end{proof}

We shall also make use of the {\it $p$-Beurling-Fourier algebras} which we define below.
As defined in 
\cite{lees,ludwigst}, a {\it weight} is a function $\ome:\what{G}\to\Ree^{>0}$
which satisfies
\[
\ome(\sig)\leq\ome(\pi)\ome(\pi')\text{ whenever }\sig\subset\pi\otimes\pi'.
\]
We will always assume that $\ome$ is {\it bounded away from zero}:
\[
\inf_{\pi\in\what{G}}\ome(\pi)>0.
\]
Notice that this is automatic if $\ome$ is {\it symmetric}, i.e.\  $\ome(\bar{\pi})=\ome(\pi)$; indeed
$\pi\otimes\bar{\pi}\supseteq 1$ so $\ome(\pi)\geq\ome(1)^{1/2}$ in this case.  We let
\begin{equation}\label{eq:wpfal}
\pfal(G,\ome)=\ell^1\text{-}\bigoplus_{\pi\in\what{G}}\ome(\pi)d_\pi^{1+\frac{1}{p'}}\smat^p_{d_\pi}.
\end{equation}
As before we give this the weighted direct sum operator space structure, with
usual interpolated structure on each $\smat^p_{d_\pi}$.  Boundedness away from zero of $\ome$ ensures
that $\pfal(G,\ome)\hookrightarrow\pfal(G)$, completely boundedly -- completely contractively
provided $\inf_{\pi\in\what{G}}\ome(\pi)\geq 1$.

\begin{corollary}\label{cor:ccBaw}
For any weight $\ome$, 
the space $\pfal(G,\ome)$ is a completely contractive Banach algebra under pointwise multiplication.
\end{corollary}

\begin{proof}
We make the obvious changes to (\ref{eq:pfaltpdual}) and hence to (\ref{eq:normgamw}).  
Then the computation follows exactly as in the proof of the last theorem.   
In place of (\ref{eq:normgamw1}) we obtain
\begin{align*}
\sup_{\pi,\pi'\in\what{G}\times\what{G}}
&\frac{1}{\ome(\pi)\ome(\pi')}\sup_{\tau\subset\pi\otimes\pi'}d_\tau^{-1/p'}
\norm{[W_{ij,\tau}]}_{\smat_{d_\tau}^{p'}} \\
&\leq \sup_{\pi,\pi'\in\what{G}\times\what{G}}\sup_{\tau\subset\pi\otimes\pi'}
\frac{1}{\ome(\tau)d_\tau^{1/p'}}\norm{[W_{ij,\tau}]}_{\smat_{d_\tau}^{p'}} \\
&= \sup_{\tau\in\what{G}}\frac{1}{\ome(\tau)d_\tau^{1/p'}}\norm{[W_{ij,\tau}]}_{\smat_{d_\tau}^{p'}}
=\norm{[W_{ij}]}_{\mat_n(\pfal(G,\ome)^*)}.
\end{align*}
\end{proof}

Let us close this section by noting that the algebras $\fal(G)$ are an interpolation scale,
$\pfal(G)=[\fal^\infty(G),\fal(G)]_{1/p}$.  Let us explain this fact, briefly, and generalize it.
Let $\ome$ and $\til{\ome}$ each be weights on $\what{G}$.  We remark that complex interpolation
is isometrically stable for completely contractively complemented subspaces.  
See, for example, \cite[2.7.6]{pisier1}.  Hence,
using (\ref{eq:wpfal}) we see that
\begin{align}\label{eq:interpolatedwpfal}
[\fal^\infty(G,\ome),\fal(G,\til{\ome})]_{1/p}
&=\ell^1\text{-}\bigoplus_{\pi\in\what{G}}d_\pi\left[d_\pi\ome(\pi)\smat^\infty_{d_\pi},
\til{\ome}(\pi)\smat^1_{d_\pi}\right]_{1/p} \notag \\
&=\ell^1\text{-}\bigoplus_{\pi\in\what{G}}d_\pi^{1+\frac{1}{p'}}
\ome(\pi)^{1/p'}\til{\ome}(\pi)^{1/p}\smat^p_{d_\pi} \\
&=\pfal(G,\ome^{1/p'}\til{\ome}^{1/p}). \notag
\end{align}
It is evident that positive powers of single weights, and products of multiple weights, remain weights.

\subsection{Spectrum of $\pfal(G)$}\label{ssec:spectrum}
We first find it desirable to compute the spectra of a certain class of Beurling-Fourier algebras.
Let $\alp>0$ and $d^\alp:\what{G}\to\Cee$ be the $\alp$-power of the dimension
weight: $d^\alp(\pi)=d_\pi^\alp$.

We shall also make use of the basic polynomial weights for a Lie group $G$.  In this case $\what{G}$
is finitely generated, i.e.\ there is a finite $S\subset\what{G}$ for which $\bigoplus_{\pi\in S}\pi$
is faithful.  We may and shall suppose that $S$ is symmetric:  
$\pi\in S$ implies $\bar{\pi}\in S$.
Let $S^{\otimes n}=\{\sig\in\what{G}:\sig\subset\pi_1\otimes\dots\otimes\pi_n:\pi_1,\dots,\pi_n\in S\}$
and $S^{\otimes 0}=\{1\}$.
Then $S$ generates $\hat{G}$ in the sense that $\what{G}=\bigcup_{n=1}^\infty S^{\otimes n}$.
See \cite[\S 4.2]{ludwigst} for details.  It is easy to check that  $\tau_S(\sig)=\min\{n:\sig\in S^{\otimes n}\}$
is subadditive.   We let $\ome_S^\alp(\pi)=(1+\tau_S(\pi))^\alp$, for $\alp\geq 0$,
and call these the {\it polynomial weights}.
Given another symmetric generating set $S'$, it is each to see that for some constants $c,C$
that $c\ome_S\leq \ome_{S'}\leq C\ome_S$, so all polynomial weights are equivalent.

Suppose now that $G$ is a general compact group.  For any finite symmetric set $S\subset\what{G}$
we may consider $\tau_S:\langle S\rangle\to\Ree^{\geq 0}$
to be defined as above, and hence define a weight $\ome_S^\alp:\langle S\rangle\to\Ree^{>0}$, as above.
Notice that $\ome_S^\alp$ is really a weight on $\what{G/N_S}\circ q$ where 
$N_S=\bigcup_{\pi\in S}\ker\pi$ and $q:G\to G/N_S$ is the quotient map.
A weight $\ome$ on a general compact group $G$ is called
{\it weakly polynomial} if for any $S$ as above, there is a constant $C_S$ and $\alp_S\geq 0$ for which
\[
\ome|_{\langle S\rangle}\leq C_S\ome_S^{\alp_S}.
\]  
(For a connected group, this was termed a 
``polynomial weight" in \cite[Def.\ 5.1]{ludwigst}.  Compare the results \cite[(5.5) \& Thm.\ 5.4]{ludwigst}.)
Note that if $G$ is totally disconnected, then any weight is weakly polynomial.  Indeed, any
finite subset $S$ of $\what{G}$ has that $N_S$ is open, so $\langle S\rangle$ is finite.

\begin{proposition}\label{prop:specfalpol}
If $\ome$ is a symmetric weakly polynomial weight on $\what{G}$, then $\fal(G,\ome)$ has spectrum $G$.
\end{proposition}

\begin{proof}
In \cite[Prop.\ 5.5]{ludwigst} this is shown for connected groups.  Let us adapt the proof for general $G$.
Fix $\pi$ in $\what{G}$, and let $S=\{\pi,\bar{\pi}\}$.  Then for any $\sig\subset\pi^{\otimes n}$,
the definition of $\tau_S$ provides that $\tau_S(\sig)\leq n$, hence $\ome_S(\sig)\leq
1+n$.   Hence using our assumption that $\ome$ is weakly polynomial we have
\[
\left(\sup_{\sig\subset\pi^{\otimes n}}\ome(\sig)\right)^{1/n}
\leq \left(\sup_{\sig\subset\pi^{\otimes n}}C_S\ome_S(\sig)^{\alp_S}\right)^{1/n}
\leq C_S^{1/n}(1+n)^{\alp_S/n}\overset{n\to\infty}{\longrightarrow} 1.
\]
The result now follows from \cite[Prop.\ 4.19]{ludwigst}.
\end{proof}

If $\ome$ is a weight on $\what{G}$ and $H$ is a closed subgroup, we follow
\cite[\S 4.1]{ludwigst} or \cite[Prop.\ 3.5]{lees} and define the {\it restricted weight} 
$\ome_G|_H$ on $\what{H}$ by
\begin{equation}\label{eq:restweight}
\ome_G|_H(\sig)=\inf\{\ome(\pi):\pi\in\what{G},\sig\subset\pi|_{H}\}
\end{equation}
where $\pi|_H$ refers to the restricted representation.

\begin{remark}\label{rem:regularity}
For a connected Lie group $G$, and a weakly polynomial weight,
$\fal(G,\ome^\alp_S)$ is shown to be regular in \cite[Thm.\ 5.11]{ludwigst}.  
We do not know how to extend this result to non-connected groups.  
For us to do this, it would be sufficient to see that for a Lie group $G$, with symmetric generating
set $S$ for $\what{G}$, that $(\ome_S)_G|_{G_e}$ is weakly polynomial on $\what{G_e}$.
We observe, however, that the easy estimate to show is in the wrong direction.

Let $S_e=\{\sig\in\what{G_e}:\sig\subset\pi|_{G_e}\text{ for some }\pi\text{ in }S\}$.
If $\sig\subset\pi|_{G_e}$, where $\tau_S(\pi)=n$, then there are $\pi_1,\dots,\pi_n$ in $S$
for which $\sig\subset\pi_1\otimes\dots\otimes\pi_n|_{G_e}=\pi_1|_{G_e}\otimes\dots\otimes\pi_n|_{G_e}$.
It follows that $\tau_{S_e}(\sig)\leq n=\tau_S(\pi)$.  Hence
\begin{equation}\label{eq:polyest}
\ome_{S_e}(\sig)\leq(\ome_S)_G|_{G_e}(\sig).
\end{equation}
\end{remark}

We shift our attention to dimension weights.  We can obtain regularity.

\begin{proposition}\label{prop:specfald}
{\bf (i)} The Gelfand spectrum of $\fal(G,d^\alp)$ is $G$.

{\bf (ii)} The algebra $\fal(G,d^\alp)$ is regular on $G$.
\end{proposition}

\begin{proof}
(i) For $\alp=1$, this result is stated as \cite[Cor.\ 5.6]{ludwigst} and erroneously attributed to \cite{parks}.
Regrettably, the proof of \cite[Cor.\ 5.6]{ludwigst} only deals with the case of connected groups.
It is sufficient to see that that the dimension weight $d$ on $G$ is a weakly polynomial weight,
and then appeal to  Proposition \ref{prop:specfalpol}.

First, suppose that $G$ is a connected Lie group.  
Then \cite[Ex.\ 5.2]{ludwigst} shows that $d$ is a weakly polynomial weight.  (This uses
ideas we shall use in proving Theorem \ref{theo:opalg}.)

Now suppose that $G$ is a Lie group, so the connected component of the identity $G_e$
is open, hence the index $[G:G_e]$ is finite.  Then for $\pi$ in $\what{G}$ and $\sig$ in $\what{G_e}$
the Frobenius reciprocity formula of \cite[Thm.\ 8.2]{mackey} (see also
\cite[2.61]{kaniutht}) yields equality of multiplicities
\[
m(\sig:\pi|_{G_e})=m(\pi:\sig\!\uparrow^G)
\]
where  $\sig\!\uparrow^G$ denotes the induced
representation.  Thus if $\sig\subset\pi|_{G_e}$, then $\pi\subset\sig\!\uparrow^G$, so
\[
d_\pi\leq d_{\sig\uparrow^G}=[G:G_e]d_\sig.
\]
Fix any symmetric generating set $S$ of $\what{G}$.  Then for any $\pi$ in $\what{G}$ and
$\sig$ in $\what{G_e}$ with $\sig\subset\pi|_{G_e}$ we have
\[
d_\pi\leq [G:G_e] d_\sig \leq C \ome_{S_e}^{\alp'}(\sig)\leq C(\ome_S)^{\alp'}_G|_{G_e}(\sig)
\leq C\ome_S^{\alp'}(\pi)
\]
for some constant $C$ and $\alp'\geq 0$, where the second
inequality follows from the fact that the dimension
weight on $G_e$ is weakly polynomial, as noted in the prior paragraph, and the fact that for any two generating sets, the respective polynomial weights are equivalent; and the third inequality is provided by (\ref{eq:polyest}).

We finally consider the case of a general compact $G$.  For any
finite symmetric set $S$ of $\what{G}$, with $N_S=\bigcap_{\pi\in S}\ker\pi$, we see that
$G/N_S$, being isomorphic to a subgroup of $\bigoplus_{\pi\in S}\pi(G)$ is a Lie group,
with $\langle S\rangle=\what{G/N_S}\circ q$, where $q:G\to G/N_S$ is the quotient map.
Hence the result of the last paragraph shows that $d|_{\langle S\rangle}\leq C_S\ome_S^{\alp_S}$
on $\langle S\rangle$, for some $C_S$ and $\alp_S$, i.e.\ $d$ is weakly polynomial.



(ii) Let us now show the regularity.
First, let us fix $\alp=1$.
We see from \cite[Thm.\ 4.1]{forrestss} that the map
$\check{\Gam}:\fal(G\times G)\to\fal(G,d)$ given by $\check{\Gam}u(s)=
\int_Gu(st,t^{-1})\,dt$ is a surjection.  If $E$ and $F$ are non-empty disjoint closed subsets
of $G$ then $\check{E}$ and $\check{F}$ are non-empty disjoint closed subsets
of $G\times G$, where $\check{S}=\{(s,t):st\in S\}$.  Then \cite[(3.2)]{eymard}
provides $u$ in $\fal(G\times G)$ such that $u|_{\check{E}}=1$ and $u|_{\check{F}}=0$.
Then $\check{\Gam}(u)|_E=1$ and $ \check{\Gam}(u)|_F=0$.

It is now straightforward to verify the (completely) isometric identification
\[
\fal(G,d^\alp)\hat{\otimes}\fal(G,d^\alp)\cong\fal(G\times G,(d\otimes d)^\alp)
\]
where $d\otimes d(\pi,\sig)=d_\pi d_\sig$.  Hence the proof of \cite[Prop.\ 2.6]{rostamis}
can be followed to show that
\[
\check{\Gam}(\fal(G\times G,(d\otimes d)^\alp))=\fal(G,d^{2\alp+1}).
\]
The recursion $\alp_0=0,\; \alp_{n+1}=2\alp_n+1$ admits solution $\alp_n=2^n-1$.  See, also,
\cite[p.\ 189]{lees}.
Hence, applying induction to the paragraph above yields that each algebra $\fal(G,d^{2^n-1})$
is regular.  The contractive embedding $\fal(G,d^{2^{\lceil \log_2(\alp+1)\rceil}-1})\hookrightarrow
\fal(G,d^\alp)$ yields the regularity of the latter algebra.
\end{proof}

We move from $p=1$ to all $1\leq p\leq\infty$.

\begin{proposition}\label{prop:spectrum}
The Gelfand spectrum of $\pfal(G,d^\alp)$ is $G$, and $\pfal(G,d^\alp)$ is regular on $G$.
\end{proposition}

\begin{proof}
For any positive integer $d$ the formal identity map $\smat^1_d\to\smat^p_d$ is a contraction, 
hence so too
is $d\smat^1_d\to d\smat^p_d\to d^{1/p'}\smat^p_d$.  Thus the inclusion map
$\fal(G,d^{\alp+1})\hookrightarrow\pfal(G,d^\alp)$ is a contraction.  The desired results are immediate
from Proposition \ref{prop:specfald}, above.
\end{proof}

We observe that for $s$ in $G$, the unitary $\pi(s)$ in $\smat_{d_\pi}^{p'}$ necessarily has norm
$d_\pi^{1/p'}$.  Hence by
by (\ref{eq:pfaldual}) and the result above, each element of the spectrum has norm $1$.  
Thus the choice of exponent $1+\frac{1}{p'}$ in (\ref{eq:pfal})
is minimal for allowing the space $\pfal(G)$ to be a Banach function algebra on $G$.  
Indeed, characters need necessarily be contractive.

Let $\theta:\Ree^{\geq 0}\to\Ree^{>0}$ be a non-decreasing weight 
(i.e.\ $\theta(s+t)\leq\theta(s)\theta(t)$) and let 
$\ome_\theta(\pi)=\theta(\log d_\pi)$.  Such weights are symmetric.

\begin{example}\label{ex:thetaweight}
The weight $\theta^\alp(t)=e^{\alp t}$, where $\alp>0$, gives
the dimension weight $\ome_{\theta^\alp}(\pi)=d_\pi^\alp$.  
The weight $w^\alp(t)=(1+t)^\alp$, on $\Ree^{\geq 0}$,
leads to $\ome_{w^\alp}(\pi)=(1+\log d_\pi)^\alp$.
\end{example}

We say that a weight $\ome$ on $\what{G}$ is {\it weakly dimension}
if there is $C$ and $\alp\geq 0$ for which $\ome(\pi)\leq Cd_\pi^\alp$.  Notice that
for $\theta$ as above, $\theta(s)\leq\theta(1)^{1+\lfloor s\rfloor}\leq \theta(1)e^{\log\theta(1)s}$,
from which it follows that $\ome_\theta$ is weakly dimension.

\begin{corollary}\label{cor:spectrum}
Let $\ome$ be any weight on $\what{G}$ which is weakly dimension. Then the
algebra $\pfal(G,\ome)$ has Gelfand spectrum $G$
and is regular on $G$.   In particular, this holds for $\ome_\theta$ where
$\theta$ is any non-decreasing weight on $\Ree^{\geq 0}$.
\end{corollary}

\begin{proof}
The bounded inclusions 
$\pfal(G,d^\alp)\subseteq \pfal(G,\ome)\subseteq
\pfal(G)$ give the first conclusion.  The second conclusion is immediate from the comments above.
\end{proof}

The following is a straightforward adaptation of \cite[Cor.\ 2.4]{forrestss}, 
which we leave to the reader. 

\begin{proposition}\label{prop:openabel}
{\bf (i)} Given a weight $\ome$ in $\what{G}$ and $1< p\leq \infty$, we have that
$\pfal(G,\ome)=\fal(G,\ome)$ isomorphically, if and only if $G$ admits an open abelian subgroup.

{\bf (ii)} For any weight $\ome$ which is  weakly dimension,
$\fal(G,\ome)=\fal(G)$ isomorphically, if and only if $G$ admits an open 
abelian subgroup.
\end{proposition}

\subsection{Isometric isomorphisms}\label{ssec:isomorphism}
The main theorem of \cite{walter} tells us that any isometric isomorphism
$\Phi:\fal(G)\to\fal(H)$ is of the form $\Phi u=u(s_0\varphi(\cdot))$ where $\varphi:H\to G$
is a homeomorphism which is either an isomorphism or anti-isomorphism of the groups.
In particular, $\fal(G)\cong\fal(H)$ isometrically, only if $G\cong H$  as topological groups.

With compact groups, the addition of certain weights does not change this result.
Furthermore, we can obtain this result for most indices $p$.  
We retain our convention that $G$ and $H$ denote compact groups.
The weights $\ome_\theta$ are defined at the end of the last section.

\begin{theorem}\label{theo:walter2}
Let $1\leq p\leq\infty$ with $p\not=2$.
Fix a non-decreasing weight $\theta:\Ree^{\geq 0}\to\Ree^{>0}$.
If $\pfal(G,\ome_\theta)\cong\pfal(H,\ome_\theta)$ isometrically isomorphically
then $G\cong H$ as compact groups.
\end{theorem}

\begin{proof}
Given any Banach space $E$, we let $B(E)$ denote the closed unit ball, and 
$S(E)$ the unit sphere.
We identify the spaces $\trig_\pi=d_\pi^{1+\frac{1}{p'}}\ome_\theta(\pi)\smat^p_{d_\pi}$, 
for $\pi$ in $\what{G}$,
as subspaces of $\pfal(G,\ome_\theta)$.  Then we have the following routine identification of 
sets of extreme points:
 \[
 \ext B(\pfal(G,\ome_\theta))
 =\ext B\left(\ell^1\text{-}\bigoplus_{\pi\in\what{G}}d_\pi^{1+\frac{1}{p'}}
 \ome_\theta(\pi)\smat^p_{d_\pi}\right)
 =\bigcup_{\pi\in\what{G}}\ext B\left(d_\pi^{1+\frac{1}{p'}}\ome_\theta(\pi)\smat^p_{d_\pi}\right).
 \]
We remark that we have for any $d$ in $\En$, $1<p<\infty$
\begin{gather*}
\ext B(\smat^1_d)=\{a\in S(\smat^1_d):\text{rank}a=1\},\;
\ext B(\smat^p_d)=S(\smat^p_d),\;
\text{and}\; \ext B(\smat^\infty_d)=\un(d)
\end{gather*}
where $\un(d)$ is the unitary group.
Indeed,  $\smat^1_d=\ell^2_d\otimes^\gamma\ell^2_d$ and the description of the projective
tensor product norm shows that $\ext B(\smat^1_d)$ must consist of rank one elements.
The fact that any two rank one elements $a$ and $b$ admit unitaries $u,v$ for which
$uav=b$ show that all such rank one elements are achieved.  For $1<p<\infty$, it is shown
by \cite{mccarthy} that $\smat^p_d$ is uniformly convex, hence strictly convex.  The description
of $\ext B(\smat^\infty_d)$ may be found in \cite{kadison} or \cite[Chap.\ 10]{flemingj}.
Hence we observe that each set $\ext B(\smat^p_d)$, $1\leq p\leq\infty$ is connected.  

Since $\Phi$ is an isometry, for each $u\in
\ext B\left(d_\pi^{1+\frac{1}{p'}}\ome_\theta(\pi)\smat^p_{d_\pi}\right)$ we have
$\Phi(u)\in\ext B\left(d_{\pi'}^{1+\frac{1}{p'}}\ome_\theta({\pi'})\smat^p_{d_{\pi'}}\right)$
for some $\pi'\in\what{H}$.  We observe that, given $\pi$ in $\what{G}$, the sets
$X_{\pi,\pi'}=\Phi^{-1}(\ext B(\smat^p_{d_{\pi'}}))\cap\ext B(\smat^p_{d_\pi})$, $\pi'$ in $\what{H}$,
comprise a cover of $\ext B(\smat^p_{d_\pi})$ by pairwise disjoint open sets.
Hence, by connectedness, we have that there is a unique $\pi'$ for which $X_{\pi,\pi'}\not=\varnothing$.
Thus we obtain a bijection $\hat{\Phi}:\what{G}\to\what{H}$, for which 
$\Phi(\trig_\pi)=\trig_{\hat{\Phi}(\pi)}$. Clearly $d_{\hat{\Phi}(\pi)}=d_\pi$ for each
$\pi$, so this map induces an isometry $\smat^p_{d_\pi}\to\smat^p_{d_{\hat{\Phi}(\pi)}}$.

Thanks to \cite[10.2.2 or 10.3.5]{flemingj} (based on results of \cite{kadison,marcus}) 
in the case $p=1,\infty$, and \cite{arazy} in the case $1<p<\infty$ but $p\not=2$, 
each isometry $\smat^p_{d_\pi}\to\smat^p_{d_{\hat{\Phi}(\pi)}}$
is of the form $a\mapsto uav$ or $a\mapsto ua^Tv$ where $u$ and $v$ are unitaries
and $a^T$ denotes the transpose with respect to some orthonormal basis.  Hence we see that
$\norm{\Phi u}_\fal=\norm{u}_\fal$ for $u$ in $\pfal(G,\ome_\theta)$ 
extends to an isometry $\fal(G)\to\fal(H)$ with dense range.  The structure
of $\Phi$ follows from \cite{walter}, accordingly.  \end{proof}


\subsection{Different operator space structures on $\fal^2(G)$}\label{ssec:hilbert}
The construction above gives
\[
\fal^2(G)=\ell^1\text{-}\bigoplus_{\pi\in\what{G}}d_\pi^{3/2}\smat_{d_\pi,OH}^2
\]
where the subscript $OH$ denotes the operator Hilbert space structure on each
space $\smat_{d_\pi}^2$.    However, we wish to observe that other choices
of operator space structure allow $\fal^2(G)$ to be a completely contractive Banach
algebra.

A {\it Hilbertian} operator space structure is an operator space structure
$\fH\mapsto \fH_E$ which may be assigned to any Hilbert space.
Such a structure
is called {\it homogeneous} if $\fB(\fH)=\fC\fB(\fH_E)$ isometrically for all $\fH$.
Furthermore, $\fH\mapsto \fH_E$ is {\it subquadratic} if for any projection
on $\fH$ we have $\norm{[x_{ij}]}_{\mat_m(\fH_E)}^2\leq \norm{[Px_{ij}]}_{\mat_m(\fH_E)}^2
+\norm{[(I-P)x_{ij}]}_{\mat_m(\fH_E)}^2$
for any $n$ and $[x_{ij}]$ in $\mat_n(\fH)$.  This is equivalent to saying that for (finite dimensional)
$\fH$ and $\fK$, the formal identity
\[
\mat_n(\fH_E)\oplus_2\mat_n(\fK_E)\hookrightarrow \mat_n((\fH\oplus^2\fK)_E)
\]
is a contraction for every $n$.  We note that the homogeneity assumption provides that
$(\fH\oplus^2\fK)_E=\fH_E\oplus^2\fK_E$, completely isometrically.
Finally, we will say that
$\fH\mapsto \fH_E$ is {\it subcross} if the identity map on the algebraic
tensor product of any two Hilbert spaces $\fH\otimes\fK$ extends to a complete contraction
$\fH_E\hat{\otimes}\fK_E\to(\fH\otimes_2\fK)_E$.   By duality, this is equivalent to having
the dual structure $\fH\mapsto \fH_{E^*}$ satisfy that the identity map on any $\fH\otimes\fK$
extend to a complete contraction $(\fH\otimes_2\fK)_{E^*}\to\fH_{E^*}\check{\otimes}\fK_{E^*}$.

We note that the standard homogeneous 
Hilbertian operator space structures $OH$, $C$ (column), $R$ (row), $R+C$, and $\max$ are subcross,
and that the dual structures $OH=OH^*$, $R=C^*$, $C=R^*$, $R\cap C=(R+C)^*$ and $\min=\max^*$
are subquadratic.  The structures $\max$ and $R+C$, themselves, however, are not subquadratic.
See the discussion in \cite[p.\ 81]{pisier}.  Let us also consider the interpolated structures,
$\fH_{C^p}=[\fH_C,\fH_R]_{1/p}$ for $1\leq p\leq\infty$, and likewise for row structure.  
We have that $C^2=OH=R^2$.
By interpolation these structures are each homogeneous.  Moreover, stability of the Haagerup tensor
products under interpolation gives us that 
\[
\fH_{C^p}\otimes^h\fK_{C^p}=[\fH_C\otimes^h\fK_C,\fH_R\otimes^h\fK_R]_{1/p}
=[(\fH\otimes^2\fK)_C,(\fH\otimes^2\fK)_R]_{1/p}=(\fH\otimes^2\fK)_{C^p}
\]
from which it is immediate that each $C^p$ is subcross.  We observe the duality $(C^p)^*=C^{p'}=R^p$.
Finally, subquadraticity of $C^p$ follows from that for $C,R$, and interpolation.

\begin{theorem}\label{theo:twoccBa}
Let $\fH\mapsto \fH_E$ be a subcross homogeneous operator space strucure whose dual 
structure is subquadratic.  Then the operator space 
\[
\fal^2_E(G)=\ell^1\text{-}\bigoplus_{\pi\in\what{G}}d_\pi^{3/2}\smat_{d_\pi,E}^2
\]
is a completely contractive Banach algebra under pointwise multiplication.
\end{theorem}

\begin{proof}
The proof is essentially the same as that of Theorem \ref{theo:ccBa}.  The subcross, homogeneity
and subquadratic conditions of $E$, provide, respectively, the first, second and third observations
required in that proof.
\end{proof}

Of course $\fal_{\max}^2(G)$ is simply $\fal^2(G)$ with its maximal operator space structure,
hence it is no surprise based on Theorem \ref{theo:ccBa}, that it is a completely contractive
Banach algebra.  As in Corollary \ref{cor:ccBaw}, we can see that the $2$-Beurling-Fourier algebras
$\fal^2_E(G,\ome)$, for weights $\ome$, are completely contractive Banach algebras.

The completely contractive Banach algebra $\fal_R^2(G)$ has actually been observed in
\cite{forrestss1}.  Let us review this briefly.
Let $\Gam:\fal(G\times G)\to\fal(G)$ be given by 
\begin{equation}\label{eq:gammamap}
\Gam u(s)=\int_G u(st,t)\,dt 
\end{equation}
so $\Gam$ averages elements of $\fal(G\times G)$ over left cosets of the diagonal subgroup
$\Del=\{(s,s):s\in G\}$.  For an elementary tensor, $u\otimes v\in\fal(G)\otimes\fal(G)\subset
\fal(G\times G)$, $\Gam(u\otimes v)=u\ast\check{v}$ where $\check{v}(s)=v(s^{-1})$.

We record the following fact from \cite[Thm.\ 2.1]{rostamis}, which is essentially
generalized by Theorem \ref{theo:pfaldel}, below.  See Remark \ref{rem:whenci}.

\begin{proposition}\label{prop:faldel}
Let $\fal_\Del(G)=\Gam(\fal(G\times G))$.  If we assign $\fal_\Del(G)$ the operator space
structure which makes $\Gam$ a complete quotient map, then $\fal_\Del(G)=\fal^2_R(G)$,
completely isometrically.
\end{proposition}

\section{Amenability properties}

We shall purposely limit our definitions to a commutative unital Banach algebra $\fA$, even when they
may be more broadly made.  After \cite{johnson0}, we say that $\fA$ is {\it amenable}
if there is a net $(w_i)$ in the projective tensor product $\fA\otimes^\gam\fA$
which is bounded, satisfies that $(m(w_i))$ is an approximate identity on $\fA$, and
for each $u$ in $\fA$, $(u\otimes 1-1\otimes u)w_i\overset{i}{\longrightarrow}0$
in the norm of $\fA\otimes^\gam\fA$.  This is equivalent to having  any bounded derivation
$D:\fA\to\fX^*$, where $\fX$ is a Banach $\fA$-bimodule with dual module $\fX^*$,
be inner, i.e.\ $D(u)=u\cdot f-f\cdot u$ for some $f$ in $\fX^*$.  

We say that $\fA$ is {\it weakly amenable}, if for every symmetric $\fA$-module $\fX$ we have that
the only bounded derivation from $\fA$ into $\fX$ is zero.  This definition was given in \cite{badecd},
where it was also noted that it is equivalent to seeing that the only bounded derivation from $\fA$
into the symmetric dual bimodule $\fA^*$ is inner.  We note that weak amenability implies
that $\fA$ admits no bounded point derivations.  We recall that if $\fA$ is a function algebra
on a compact Hausdorff space $X$, then a point derivation at $x$ in $X$ is a linear functional 
$d:\fA\to\Cee$ which satisfies $d(uv)=d(u)v(x)+u(x)d(v)$.

Both definitions above admit obvious analogues in the setting of completely contractive Banach
algebras, where we substitute operator projective tensor product for projective, and  study only
completely bounded derivations.  See, for example, \cite{ruan}.  It is obvious that (weak) amenability
implies operator (weak) amenability.

Let us begin with a synthesis of well-known facts.  We adopt the perspective of
\cite{forrestss} and let $\fA(G)$ be a unital (regular) Banach algebra of functions on $G$,  
contains $\trig(G)$ as a dense subspace, and which 
admits an operator space structure with respect to which it is completely contractive, and 
for which translations are complete isometries and are continuous on $G$.  
We suppose that $\fA(G)\hat{\otimes}\fA(G)$ is a (regular) function algebra on $G\times G$.
We let $\fA_\Del(G)$ denotes the image of $\fA(G)\hat{\otimes}\fA(G)$
under the map $\Gam$ of (\ref{eq:gammamap}), and let its operator space structure is given to 
make $\Gam$ a complete quotient map.

\begin{proposition}\label{prop:wachar}
The algebra $\fA(G)$ is operator weakly amenable if and only if 
$\fA_\Del(G)$ admits a bounded point derivation at $e$. 
\end{proposition}

\begin{proof}
We first note that the unital algebra $\fA(G)$ is weakly amenable if and only if
the ideal $\ideal_{\fA\hat{\otimes}\fA}(\Del)=\{u\in\fA(G)\hat{\otimes}\fA(G):u|_\Del=0\}$ satisfies
that  it is {\it essential}:  $\overline{\ideal_{\fA\hat{\otimes}\fA}(\Del)^2}
=\ideal_{\fA\hat{\otimes}\fA}(\Del)$.  Indeed see
\cite[Thm.\ 3.2]{groenbaek} (which is shown to hold in the operator space setting in
\cite{spronk}).  A more illuminating proof may be found in \cite[Thm.\ 2.2]{runde}; being
mostly functorial, this proof is straightforward to modify into the operator space setting
with help of the infinite matrix techniques of \cite[10.2.1]{effrosrB}.  

By \cite[Cor.\ 1.5]{forrestss}
having essentiality of $\ideal_{\fA\hat{\otimes}\fA}(\Del)$ is equivalent to essentiality of
$\ideal_{\fA_\Del}(e)=\{u\in\fA_\Del(G):u(e)=0\}$.  Any bounded linear functional on $\fA_\Del(G)$
which vanishes on $\ideal_{\fA_\Del}(e)^2$, but not on $\ideal_{\fA_\Del}(e)$, is a bounded
point derivation.  Thus by the Hahn-Banach theorem,  $\ideal_{\fA_\Del}(e)$ is essential if and only if
$\fA_\Del(G)$ admits a bounded point derivation at $e$.
\end{proof}

Let us now consider the nature of point derivations at $e$ on $\fA(G)$.
We wish to use the perspective of \cite{cartwrightm}.
We observe that $\trig(G)^\dagger\cong\prod_{\pi\in\what{G}}\mat_{d_\pi}$ is an
involutive algebra with the usual conjugate transpose on each $\mat_{d_\pi}$
and the coproduct $M:\trig(G)^\dagger\to\trig(G\times G)^\dagger$ is a $*$-homomorphism.
We further assume that $\fA(G)^*\subset \trig(G)^\dagger$ is closed under this involution.
In the case that $\fA(G)=\pfal(G,\ome)$, this is straightforward to check.
A point derivation $D$ at $e$ on $\fA(G)$ clearly restricts to such on $\trig(G)$.
Moreover, we have the derivation law for the coproduct is $MD=1\otimes D+D\otimes 1$.
Hence $D^*$ is also a derivation.  
Hence we may write as a linear combination of two skew-hermitian derivations:
$D=\frac{1}{2}(D-D^*)+\frac{1}{2i}(iD+iD^*)$.  Then \cite[Thm.\ 1]{cartwrightm}
provides us with the following.

\begin{proposition}\label{prop:derchar}
Let $G$ be a connected Lie group and $\fA(G)$ satisfy all the conditions above.  
Then each if there exists any bounded point derivation on $\fA(G)$, then there is necessarily a 
bounded skew-symmetric point derivation.  Furthermore, each
bounded skew-symmetric point derivation $D$ on
$\fA(G)$ is a classical Lie derivative, i.e.\ of the form
$D(u)=\left.\frac{d}{d\theta}\right|_{\theta=0}u(t_\theta)$ where 
$\theta\mapsto t_\theta:\Ree\to G$ is a one-parameter subgroup.
\end{proposition}

Let recall how $\fA_\Del(G)$ witnesses the operator amenability of $\fA(G)$.  
We let $\fA(G)$ satisfy all of the assumptions we have so far.  

\begin{proposition}\label{prop:johnson}
A completely contractive Banach algebra $\fA(G)$ is operator amenable
if and only if $\ideal_{\fA_\Del}(e)=\{u\in\fA_\Del(G):u(e)=0\}$ admits a bounded approximate
identity.
\end{proposition}

\begin{proof}
A splitting result of \cite{helemskii} (see \cite{curtisl}) shows that $\fA(G)$ is
operator amenable exactly when $\ideal_{\fA\hat{\otimes}\fA}(\Del)=\{w\in
\fA(G)\hat{\otimes}\fA(G):w(s,s)=0\text{ for all }s\text{ in }G\}$ admits a bounded
approximate identity.  While this result
is stated in the Banach space category, its proof moves easily to operator spaces.
Then \cite[Sec.\ 1.2 \& Cor.\ 1.5]{forrestss1} shows that the latter statement is equivalent
to $\ideal_{\fA_\Del}(e)$ admitting a bounded approximate identity.  There is also
a beautifully ``hands-on" proof of this in \cite[Thm.\ 3.2]{johnson}, which can be easily modified to
operator spaces and to this general setting.
\end{proof}

\subsection{Operator weak amenability and operator amenability}
\label{ssec:opamenprop}
Motivated by the considerations above, we may immediately 
launch into the first major result of this section.  As above we let
$\pfal_\Del(G,\ome)=\Gam(\pfal(G,\ome)\hat{\otimes}\pfal(G,\ome))$.

\begin{theorem}\label{theo:pfaldel}
The space $\pfal_\Del(G,\ome)$, qua quotient of $\pfal(G,\ome)\hat{\otimes}\pfal(G,\ome)$ 
via $\Gam$, is given by
\begin{gather*}
\pfal_\Del(G,\ome)= \fal^{r(p)}(G,d^{\beta(p)}\Omega) \\
\text{where }\frac{1}{r(p)}+\frac{|p-2|}{2p}=1,\;
\beta(p)=\begin{cases}4-\frac{4}{p} & \text{if }1\leq p<2 \\
 2 &\text{if }p\geq 2\end{cases}\quad\text{and }\Omega(\pi)=\ome(\pi)\ome(\bar{\pi}).
\notag
\end{gather*}
\end{theorem}

We call $\Ome$ the {\it symmetrization} of $\ome$.

\begin{remark} \label{rem:whenci}
We remark that the identification is not, in general, a complete isometry.  
Indeed the second identification in (\ref{eq:cbintcol}), below, 
is not generally a complete isometry.  However if $p=1$, then we obtain the 
$\smat^2_{d,C}$ in (\ref{eq:cbintcol}), from which we can deduce Proposition \ref{prop:faldel}, above.
We leave the details to the reader; or see \cite[Thm.\ 2.1]{rostamis}
\end{remark}

\begin{proof}
We first note that 
\[
\Gam^*(\lam(s))=\int_G\lam(st)\otimes\lam(t)\,dt=(\lam(s)\otimes I) \int_G\lam(t)\otimes\lam(t)\,dt
\]
where the integral is a weak* integral in $(\pfal(G)\hat{\otimes}\pfal(G))^*$.
Since $\trig(G)$ is dense in $\pfal(G,\ome)$, and $\spn\lam(G)$ is weak$^\dagger$-dense
in $\trig(G)^\dagger$, we have for $T$ in $\trig(G)^\dagger$ that
\begin{equation}\label{eq:gamadj}
\Gam^*(T)=(T\otimes I )\int_G\lam(t)\otimes\lam(t)\,dt.
\end{equation}
Using the decomposition $\lam=\bigoplus_{\pi\in\what{G}}\pi$,  the Schur orthogonality 
relations tell us that $\int_G \lam(s)\otimes\lam(s)\,ds\cong\bigoplus_{\pi,\pi'\in \what{G}\times\what{G}}
\int_G \pi(s)\otimes\pi'(s)\,ds$, where each integral is zero unless $\pi'=\bar{\pi}$.
Thus for $T=(T_\pi)_{\pi\in\what{G}}$ we obtain
\[
\Gam^*(T)=\bigoplus_{\pi\in\what{G}}
(T_\pi\otimes I) \int_G\pi(t)\otimes\bar{\pi}(t)\,dt.
\]
Further, select for each $\pi$ a basis $\{e^\pi_1,\dots,e^\pi_{d_\pi}\}$ for $\fH_\pi$,
and let $e_{ij}^\pi
=e^\pi_i\otimes e^\pi_j$ in $\smat^{p'}_{d_\pi}\cong\col^{p'}_{d_{\pi}}\otimes^h\row^{p'}_{d_{\pi}}$.
For each $\pi$, we consider both $\pi$ and $\bar{\pi}$ as acting on one and the same Hilbert space,
and we depict their matrices with the same basis.
Then a more refined use of the Schur orthogonality relations shows that 
\[
\int_G\pi(t)\otimes\bar{\pi}(t)\,dt=\frac{1}{d_\pi}\sum_{i,j=1}^{d_\pi}e_{ij}^\pi\otimes e_{ij}^{\bar{\pi}}.
\]
We write each $T_\pi=\sum_{k,l=1}^{d_\pi}T_{\pi,kl}e_{kl}^\pi$ to obtain
\begin{align*}
(T_\pi\otimes I) \int_G\pi(t)\otimes\bar{\pi}(t)\,dt
&=\frac{1}{d_\pi}\sum_{i,j,k=1}^{d_\pi}T_{\pi,ki}e^\pi_{kj}\otimes e_{ij}^{\bar{\pi}} \\
&\cong\frac{1}{d_\pi}\sum_{i,j,k=1}^{d_\pi}T_{\pi,ki}e^\pi_k\otimes e^\pi_j\otimes e_i^\pi\otimes e^\pi_j.
\end{align*}
Consider the sequence of maps applied to each of the elements above:

\begin{tabular}{ll}
$\displaystyle\sum_{i,j,k=1}^{d_\pi}T_{\pi,ki}e^\pi_k\otimes e^\pi_j\otimes e_i^\pi\otimes e^\pi_j$
&$\displaystyle\in (\col^{p'}_{d_\pi}\otimes^h\row^{p'}_{d_\pi})\check{\otimes}
(\col^{p'}_{d_\pi}\otimes^h\row^{p'}_{d_\pi})$ \\
$\displaystyle\quad\mapsto 
\sum_{i,j,k=1}^{d_\pi}T_{\pi,ki}e^\pi_k\otimes e^\pi_j\otimes e_i^\pi\otimes e^\pi_j$
&$\displaystyle\in \col^{p'}_{d_\pi}\check{\otimes}\row^{p'}_{d_\pi}\check{\otimes}
\col^{p'}_{d_\pi}\check{\otimes}\row^{p'}_{d_\pi}$ \\
$\displaystyle\quad\mapsto
\sum_{k,i=1}^{d_\pi}T_{\pi,ki}e^\pi_k\otimes e_i^\pi\otimes
\sum_{j=1}^{d_\pi} e^\pi_j\otimes e^\pi_j $
&$\displaystyle\in \col^{p'}_{d_\pi}\check{\otimes}\col^{p'}_{d_\pi}\check{\otimes}
\row^{p'}_{d_\pi}\check{\otimes}\row^{p'}_{d_\pi}$ \\
$\displaystyle\quad\mapsto
\sum_{k,i=1}^{d_\pi}T_{\pi,ki}e^\pi_k\otimes e_i^\pi\otimes
\sum_{j=1}^{d_\pi} e^\pi_j\otimes e^\pi_j $
&$\displaystyle\in (\col^{p'}_{d_\pi}\check{\otimes}\col^{p'}_{d_\pi})\otimes^h
(\row^{p'}_{d_\pi}\check{\otimes}\row^{p'}_{d_\pi})$ \\
$\displaystyle\quad\mapsto
\sum_{i,j,k=1}^{d_\pi}T_{\pi,ki}e^\pi_k\otimes e^\pi_j\otimes e_i^\pi\otimes e^\pi_j$
&$\displaystyle\in (\col^{p'}_{d_\pi}\otimes^h\row^{p'}_{d_\pi})\check{\otimes}
(\col^{p'}_{d_\pi}\otimes^h\row^{p'}_{d_\pi}).$
\end{tabular}
\newline Each map is a contraction on the  elements to which it is applied.
The map at the second from last line is contractive by virtue of the fact that it is
being applied to an elementary tensor and all tensor norms are cross-norms. The map in the last line
is a contraction thanks to the shuffle theorem of \cite{effrosr1}.  We now observe that
we have isometric identifications
\begin{equation}\label{eq:cbintcol}
\col^{p'}_d\check{\otimes}\col^{p'}_d\cong\fC\fB(\col^p_d,\col^{p'}_d)=
\smat^{\frac{2pp'}{|p-p'|}}_d
\end{equation}
where the first identification is standard, and the second is observed in \cite[Lem.\ 5.9]{xu}.
We further note that $\frac{2pp'}{|p-p'|}=\frac{2p}{|p-2|}$.  Since $(\row_d^{p'})^*\cong\col_d^p$, we 
likewise obtain
$\row^{p'}_d\check{\otimes}\row^{p'}_d\cong\smat^{\frac{2p}{|p-2|}}_d$, as well.
Collecting together the results of the prior three observations we obtain
\begin{align*}
\frac{1}{d_\pi}\norm{\sum_{i,j,k=1}^{d_\pi}T_{\pi,ki}e^\pi_{kj}\otimes e_{ij}^\pi 
}_{\smat^{p'}_{d_\pi}\check{\otimes}\smat^{p'}_{d_\pi}}
&=\frac{1}{d_\pi}\norm{T_\pi\otimes I}_{\smat_{d_\pi}^{\frac{2p}{|p-2|}}\otimes^h
\smat_{d_\pi}^{\frac{2p}{|p-2|}}} \\
&=d_\pi^{\frac{|p-2|}{2p}-1}\norm{T_\pi}_{\smat_{d_\pi}^{\frac{2p}{|p-2|}}}.
\end{align*}
We then observe, in analogy to (\ref{eq:pfaldualt}), that
\[
(\pfal(G,d^\alp)\hat{\otimes}\pfal(G,d^\alp))^*\cong
\underset{\pi,\pi'\in\what{G}}{\ell^\infty\text{-}\bigoplus}
\frac{(d_\pi d_{\pi'})^{-\frac{1}{p'}}}{\ome(\pi)\ome(\pi')}
\smat^{p'}_{d_\pi}\check{\otimes}\smat^{p'}_{d_{\pi'}}.
\]
In summary, for $T$ in $\trig(G)^\dagger$ we find that 
\[
\Gam^*(T)\in(\pfal(G,\ome)\hat{\otimes}\pfal(G,\ome))^*
\quad\Leftrightarrow\quad
\sup_{\pi\in\what{G}}
\frac{d_\pi^{-\frac{2}{p'}+\frac{|p-2|}{2p}-1}}{\ome(\pi)\ome(\bar{\pi})}
\norm{T_\pi}_{\smat_{d_\pi}^{\frac{2p}{|p-2|}}}<\infty.
\]
Hence we obtain
\begin{equation}\label{eq:pfaldeldual}
\pfal_\Del(G,\ome)^*\cong\ell^\infty\text{-}\bigoplus_{\pi\in\what{G}} 
\frac{d_\pi^{-\frac{2}{p'}+\frac{|p-2|}{2p}-1}}{\ome(\pi)\ome(\bar{\pi})}\smat_{d_\pi}^{\frac{2p}{|p-2|}}
=\ell^\infty\text{-}\bigoplus_{\pi\in\what{G}} 
\frac{d_\pi^{-\beta(p)-\frac{1}{r(p)'}}}{\Omega(\pi)}\smat_{d_\pi}^{r'(p)}
\end{equation}
where $r(p)'=\frac{2p}{|p-2|}$ and $\beta(p)=1+\frac{2}{p'}-\frac{|2-p|}{p}$. 
The desired result follows by duality.
\end{proof}

In particular, we obtain an isometric identification
\[
\fal^2_\Del(G)=\fal(G,d^2).
\]
This stands in contrast to \cite[Thm.\ 2.6]{forrestss} where the isometric identification
\[
\fal^2_{R,\Del}(G)=\fal^2(G,d)
\]
is obtained.  Let us compare these two further.  In the notation of Section \ref{ssec:hilbert}
we let for $1\leq q\leq\infty$,
$\fal^2_{R^q,\Del}(G,\ome)=\Gam(\fal^2_{R^q}(G,\ome)\hat{\otimes}\fal^2_{R^q}(G,\ome))$.


\begin{theorem}\label{theo:twofaldel}
The space $\fal^2_{R^q,\Del}(G,\ome)$, qua quotient of 
$\fal^2_{R^q}(G,\ome)\hat{\otimes}\fal^2_{R^q}(G,\ome)$ via $\Gam$, is given by
\begin{gather*}
\fal^2_{R^q,\Del}(G,\ome)= \fal^{r(q)}(G,d^{\gam(q)}\Omega) \\
\text{where }\frac{1}{r(q)}+\frac{|q-2|}{2q}=1,\;
\gam(q)=\begin{cases}3-\frac{2}{q} & \text{if }1\leq q<2 \\
 1+\frac{2}{q} &\text{if }q\geq 2\end{cases}\quad\text{and }\Omega(\pi)=\ome(\pi)\ome(\bar{\pi}).
\notag
\end{gather*}
\end{theorem}

\begin{remark}
As before, we do not generally obtain a complete isometry.  However, when $q=\infty$ we will, in fact,
get $\fal^2_{R,\Del}(G,\ome)=\fal^2_R(G,d\Ome)$ completely isometrically, as we comment in the proof,
below.  In particular, we generalize and refine \cite[Thm.\ 2.1]{forrestss1}:  we see that
$\fal^2_{R,\Del}(G,d^{2^n-1})=\fal^2_R(G,d^{2^{n+1}-1})$, completely isometrically.
\end{remark}

\begin{proof}
We use the template of the proof of Theorem \ref{theo:pfaldel}.  First recall that we get isometric
identifications 
\[
\smat^2_{d,R^{q'}}=\row_d^{q'}\otimes^h \row_d^{q'}\quad\text{and}\quad
\row_d^{q'}\check{\otimes} \row_d^{q'}\cong\fC\fB(\row_d^q,\row_d^{q'})
=\smat_d^{\frac{2q}{|q-2|}}
\]
which are both completely isometric if $q=1,\infty$ and we let $\smat_d^2=\smat_{d,R^{q'}}^2$.
We may then compute, as in the last proof, that for $T$ in $\trig(G)^\dagger$ and each $\pi$ in $\what{G}$
we have 
\begin{align*}
&\norm{\Gam^*(T)_{\pi,\bar{\pi}}}_{\smat^2_{d_\pi,R^{q'}}\check{\otimes}\smat^2_{d_\pi,R^{q'}}}
=\frac{1}{d_\pi}\norm{\sum_{i,j,k=1}^{d_\pi}T_{\pi,ki}e^\pi_{kj}\otimes e_{ij}^\pi
}_{\smat^2_{d_\pi,R^{q'}}\check{\otimes}\smat^2_{d_\pi,R^{q'}}} \\
&\phantom{mm}=\frac{1}{d_\pi}
\norm{\sum_{k,i=1}^{d_\pi}T_{\pi,ki}e^\pi_k\otimes e_i^\pi\otimes
\sum_{j=1}^{d_\pi} e^\pi_j\otimes e^\pi_j}_{(\row_{d_\pi}^{q'}\check{\otimes} \row_{d_\pi}^{q'})\otimes^h
(\row_{d_\pi}^{q'}\check{\otimes} \row_{d_\pi}^{q'})} \\
&\phantom{mm}=\frac{1}{d_\pi}\norm{T_\pi\otimes I}_{\smat_{d_\pi}^{\frac{2q}{|q-2|}}\otimes^h
\smat_{d_\pi}^{\frac{2q}{|q-2|}}}
=d_\pi^{\frac{|q-2|}{2q}-1}\norm{T_\pi}_{\smat_{d_\pi}^{\frac{2q}{|q-2|}}}.
\end{align*}
Hence we see for such $T$ that
\[
\Gam^*(T)\in(\fal^2_{R^q}(G,\ome)\hat{\otimes}\fal^2_{R^q}(G,\ome))^*
\quad\Leftrightarrow\quad
\sup_{\pi\in\what{G}}
\frac{d_\pi^{-1+\frac{|q-2|}{2q}-1}}{\ome(\pi)\ome(\bar{\pi})}
\norm{T_\pi}_{\smat_{d_\pi}^{\frac{2q}{|q-2|}}}<\infty.
\]
Hence we obtain
\[
\fal^2_{R^q,\Del}(G,\ome)^*\cong\ell^\infty\text{-}\bigoplus_{\pi\in\what{G}} 
\frac{d_\pi^{-2+\frac{|q-2|}{2q}}}{\ome(\pi)\ome(\bar{\pi})}\smat_{d_\pi}^{\frac{2p}{|p-2|}}
=\ell^\infty\text{-}\bigoplus_{\pi\in\what{G}} 
\frac{d_\pi^{-\gam(p)-\frac{1}{r'}}}{\Omega(\pi)}\smat_{d_\pi}^{r'}
\]
where $r(q)'=\frac{2q}{|q-2|}$ and $\gam(q)=2-\frac{|2-q|}{q}$. 
If $q=1,\infty$, then $r(q)=r(q)'=2$ and the description above is completely isometric provided
we use structure $R^{q'}$ on each $\smat_{d_\pi}^2$.
\end{proof}

Observe that if $\ome$ is a weakly dimension weight then it follows from
Corollary \ref{cor:spectrum} and
\cite{tomiyama}, $\pfal(G,\ome)\hat{\otimes}\pfal(G,\ome)$ has spectrum $G\times G$, and is regular on 
its spectrum.  Hence the analysis of the beginning of the section applies.

We can now apply our results to the $2\times 2$ special unitary group $\su(2)$.
Observe that in the case $p=1$, we improve a result of \cite[Thm.\ 4.2 (v)]{lees}.

\begin{theorem}\label{theo:sutwo}
For $p\geq 1$, the completely contractive Banach algebra $\pfal(\su(2),d^\alp)$ admits a bounded 
point derivation if and only if $\alp\geq 1$.  The algebra $\pfal(\su(2),d^\alp)$ is operator weakly amenable
if and only if 
\[
1\leq p<\frac{4}{3+2\alp}.
\]
In particular, if $p=1$, $\fal(\su(2),d^\alp)$ is operator weakly amenable
if and only if $\alp<\frac{1}{2}$.
\end{theorem}

\begin{proof}
Up to conjugation, the only one-parameter subgroups in $G=\su(2)$ are 
$\theta\mapsto t_\theta=\mathrm{diag}(e^{i\theta},e^{-i\theta})$. 
By Proposition \ref{prop:wachar},
$D(u)=\left.\frac{d}{d\theta}\right|_{\theta=0}u(t_\theta)$ is, up to conjugacy and scalar,
the only candidate bounded skew-symmetric point derivation at $e$ on either of $\pfal(G,d^\alp)$
or on $\pfal_\Del(G)$ and hence its boundedness determines the only occasions in which such
a derivation may exist.  

We recall that $\what{\su}(2)=\{\pi_n:n=0,1,2,\dots\}$ and $d_{\pi_n}=n+1$,
and $\pi_n(t_\theta)=\mathrm{diag}(e^{in\theta},e^{i(n-2)\theta},\dots,e^{-i(n-2)\theta},e^{-in\theta})$.
Thus in the dual pairing (\ref{eq:dualpairing}) we have for $u$ in $\trig(\su(2))$ that
\begin{align*}
\langle u,D\rangle&=\sum_{n=0}^\infty(n+1)\trace
\left(\hat{u}(\pi_n)\left.\frac{d}{d\theta}\right|_{\theta=0}
\mathrm{diag}(e^{in\theta},e^{i(n-2)\theta},\dots,e^{-i(n-2)\theta},e^{-in\theta})\right) \\
&=\sum_{n=0}^\infty (n+1)\trace\bigl(\hat{u}(\pi_n)[i\mathrm{diag}(n,n-2,\dots,-(n-2),-n)]\bigr)
\end{align*}
so $D_{\pi_n}=i\mathrm{diag}(n,n-2,\dots,2-n,-n)$.
Hence for any $s\geq 1$ we find
\[
\norm{D_{\pi_n}}_{\smat^s_{n+1}}=\left(\sum_{j=0}^n|n-2j|^s\right)^{1/s}.
\]
Elementary integral estimates yield constants $c, C$ (depending only on $r'$) for which
\begin{equation}\label{eq:polysum}
c(n+1)^{1+\frac{1}{s}}\leq \left(\sum_{j=0}^n|n-2j|^s\right)^{1/s}\leq C(n+1)^{1+\frac{1}{s}}.
\end{equation}

Thus setting $s=p'$, we see from the weighted analogue of (\ref{eq:pfaldual})
that $D\in\pfal(G,d^\alp)^*$ if and only if 
\[
\sup_{n=0,1,2,\dots}(n+1)^{-(\alp+\frac{1}{p'})+(1+\frac{1}{p'})}<\infty
\]
i.e.\ $-\alp+1\leq 0$, so $\alp\geq 1$.  This characterizes when we get a bounded point derivation 
at $e$.  Since $\pfal(G,d^\alp)$ is evidently isometrically translation invariant, we get a bounded point
derivation at any point if and only if we get one at $e$.

We now wish to  determine operator weak amenability.
Hence we need to establish those $p$ which allow $D\in(\pfal_\Del(G))^*$.  
We use (\ref{eq:pfaldeldual}), with $\Omega(\pi_n)=d_{\pi_n}^{2\alp}=(n+1)^{2\alp}$,
and $s=r(p)'=\frac{2p}{|p-2|}$.  We see that $D\in(\pfal_\Del(G))^*$ if and only if 
\[
\sup_{n=0,1,2,\dots}(n+1)^{-(\beta(p)+2\alp+\frac{1}{r(p)'})+(1+\frac{1}{r(p)'})}<\infty
\]
i.e.\ exactly when $1-\beta(p)-2\alp\leq 0$.
For $p\geq 2$ this always holds; for $1\leq p< 2$ we obtain 
$-3+\frac{4}{p}\leq 2\alp$, which gives the desired result.
\end{proof}

With less precise data we can gain a more general result, but with weaker control.

\begin{theorem}\label{theo:liewa} Let $\alp\geq 0$.
Let $G$ be a non-commutative connected Lie group.  Then $\pfal(G,d^\alp)$ admits
a bounded point derivation if $\alp\geq 1$ and  is not operator weakly amenable if 
$p\geq\frac{4}{3+2\alp}$.  Moreover, $\fal^2_{R^q}(G)$ is never operator weakly amenable.
\end{theorem}

\begin{proof}
According to \cite{plymen}, there is a one-parameter subgroup which, for each $\pi$ in $\what{G}$ 
satisfies $\theta\mapsto \pi(t_\theta)=\mathrm{diag}(e^{in_1\theta},\dots,e^{in_{d_\pi}\theta})$
(up to conjugacy) where each $|n_j|<d_\pi$.  Thus we obtain the following estimates
for the point derivation given by $D(u)=\left.\frac{d}{d\theta}\right|_{\theta=0}u(t_\theta)$:
\[
\norm{D_\pi}_{\smat_{d_\pi}^s}=\left(\sum_{j=1}^{d_\pi} |n_j|^s\right)^{1/s}\leq (d_\pi d_\pi^s)^{1/s}=
d_\pi^{1+\frac{1}{s}}.
\]
The same estimates as above show that $D\in\pfal(G,d^\alp)$ for $\alp\geq 1$ and
$D\in(\pfal_\Del(G,d^\alp))^*$ for $p\geq\frac{4}{3+2\alp}$.

Now we set $p=2$, $\alp=0$, but deform the constituent Hilbert-Schmidt matrices.  Thanks to
Theorem \ref{theo:twofaldel} we have that the derivation above defines an element of 
$\fal^2_{R^q,\Del}(G)^*$
if the quantities $d_\pi^{-\gam(q)-\frac{1}{r(q)'}}d_\pi^{1+\frac{1}{r(q)'}}$ are uniformly bounded in $\pi$.  
Since the dimensions are unbounded -- see, e.g.\ \cite{moore} -- this requires  that $\gam(q)\geq 1$
which is always satisfied.
\end{proof}

\begin{remark}
Given the weight $\ome_{\alp,\beta}(\pi)=d_\pi^\alp(1+\log d_\pi)^\beta$ on $\what{G}$, where
$\alp,\beta\geq 0$.  It is easy to check that the addition of the weight $\ome_{w^\beta}(\pi)=
(1+\log\pi)^\beta$, as defined in Example \ref{ex:thetaweight},
affects neither of the proofs of the prior two theorems, hence 
affects neither of their outcomes.
\end{remark}

We close by addressing a result on operator amenability.  
Using terminology of \cite{mcmullenp}, we say that $G$ is {\it tall}
of for each fixed $d$, $\{\pi\in\what{G}:d_\pi=d\}$ is finite.  For an infinite group this means
$\lim_{\pi\to\infty}d_\pi=\infty$.  According to \cite[Thm.\ 3.2]{hutchinson}, a semisimple
compact Lie group is tall.  
However, there exist tall totally disconnected groups \cite{hutchinson1}.

\begin{theorem}\label{theo:tallnoa}
Suppose that $G$ is infinite and tall.
Then $\pfal(G)$ is not operator amenable for any $p>1$,
and $\fal^2_{R^q}(G)$ is not operator amenable for any $1\leq q\leq\infty$.
\end{theorem}

\begin{proof}
This uses the idea of \cite[Thm.\ 3.10]{lees}.
Appealing to  Theorem \ref{theo:pfaldel} and \ref{theo:twofaldel}, respectively,
we obtain  duality relations 
\begin{align*}
\pfal_\Del(G)&\cong\left(c_0\text{-}\bigoplus_{\pi\in\what{G}}
d_\pi^{-\frac{1}{r(p)'}-\beta(p)}\smat^{r(p)'}_{d_\pi}
\right)^*,\text{ and} \\
\fal^2_{R^q,\Del}(G)&\cong\left(c_0\text{-}\bigoplus_{\pi\in\what{G}}
d_\pi^{-\frac{1}{r(q)'}-\gam(q)}\smat^{r(q)'}_{d_\pi}
\right)^*.
\end{align*}
We observe, moreover, that $\beta(p)>0$ if $p>1$, and $\gam(q)\geq 1>0$ for $1\leq q\leq\infty$.
Hence, if $G$ is tall, each $\lam(s)=(\pi(s))_{s\in G}$ is an element of this predual of 
$\pfal_\Del(G)$, respectively $\fal^2_{R^p,\Del}(G)$.  

Now suppose $\pfal(G)$ is operator amenable.
Then the bounded approximate identity for $\ideal_{\pfal_\Del}(e)$ (respectively for
$\ideal_{\fal^2_{R^p,\Del}}(e)$)
promised by Proposition \ref{prop:johnson}, admits the indicator function $1_{G\setminus\{e\}}$
as a weak*-cluster point, thanks to weak*-continuity of evaluation characters, and
to the regularity of $\pfal_\Del(G)$ (respectively of  
$\fal^2_{R^p,\Del}(G)$).  This forces $G$ to be discrete, hence finite.
\end{proof}

It is hence obvious that for tall $G$ and any weight $\ome$ on $\what{G}$,
no algebra $\pfal(G,\ome)$ is operator amenable for $p>1$.
We suspect that for $p>1$, $\pfal(G)$ is only operator amenable if and only if
$G$ admits an open abelian subgroup.     We will be able to transport Theorem \ref{theo:tallnoa}
to  connected non-abelian groups.  See Section \ref{ssec:directproduct}.  

For the case of $\fal^2_R(G)=\fal_\Del(G)$, we have that this algebra is
operator amenable exactly when $G$ is virtually abelian, and operator weakly amenable
exactly when the connected component of the identity, $G_e$, is abelian.
See \cite[Thm.\ 4.1]{forrestss1}.

\subsection{Weak amenability and amenability}  \label{ssec:amenprop}
We observe that it is easy to characterize amenability of $\pfal(G,d^\alp)$.

\begin{proposition}\label{prop:amenable}
For any $1\leq p\leq \infty$ and weakly dimension weight $\ome$, 
$\pfal(G,\ome)$ is amenable if and only if $G$ admits an open abelian 
subgroup.
\end{proposition}

\begin{proof}
If $\pfal(G,\ome)$ --- which is dense within $\fal(G)$ --- is amenable, then so too must be $\fal(G)$.
Hence $G$ contains an open abelian subgroup by \cite[Thm.\ 2.3]{forrestr} (or see \cite{runde1}).  
The converse is immediate from Proposition \ref{prop:openabel} and the fact that $\fal(G)$
is amenable in this case (see \cite{laulw}).
\end{proof}

The weak amenability of the Banach algebras $\pfal(G,d^\alp)$
is now straightforward to establish.  
It is interesting in its own right to observe that our algebras respect  quotient subgroups. 

Let $\ome_N=\ome|_{\what{G/N}\circ q}$ where $q:G\to G/N$ is the quotient map.
Notice that $\ome_N$ retains all of our assumptions of being bounded away from zero
and weakly dimension.

\begin{lemma}\label{lem:quot}
Let $N$ be a closed normal subgroup of $G$.  Then
the map $u\mapsto T_N(u)=\int_N u(\cdot n)\, dn$ is a completely contractive
projection for which we have $T_N(\pfal(G,\ome))\cong\pfal(G/N,\ome_N)$ completely isometrically.
Moreover, if $(N_\iota)$ is a decreasing net of subgroups converging to $e$
(i.e.\ any open neighborhood of the identity contains some $N_\iota$), then 
$u=\lim_{\iota}T_{N_\iota}u$ for each $u$ in $\pfal(G,\ome_N)$.
\end{lemma}

\begin{proof}
The proof of the first statement may be adapted from that of \cite[Prop.\ 4.14 (i)]{ludwigst},
with obvious changes.  In particular we see that
\[
\pfal(G/N,\ome_N)=T_N(\pfal(G,\ome))
\cong\ell^1\text{-}\bigoplus_{\pi\in\what{G/N}\circ q}d_\pi^{1+\frac{1}{p'}}\ome(\pi)
\smat^p_{d_\pi}.
\]
The second statement is well-known, see for example the proof of
\cite[Thm.\ 3.3]{forrestr}. Briefly, given $u$ in $\pfal(G,\ome)$, by continuity of translations we 
can arrange $N_\iota$ so small that $\norm{u-u(\cdot n)}_{\pfal(G,\ome)}$ is uniformly small for $n$ in 
$N_\iota$.  Averaging over $N_\iota$ does not increase this norm.
\end{proof}

We shall make repeated use the observation  that if a commutative Banach algebra has within itself
a dense image of a weakly amenable commutative algebra, then it is weakly amenable;
see \cite[Def.\ 1.1 \& Thm.\ 1.5]{badecd}.  The result (ii), below, improves upon
\cite[Thm.\ 3.14]{lees}, where operator weak amenability is established.

\begin{proposition}
{\bf (i)}  For any $1\leq p\leq \infty$ and weight $\ome$ weakly dominated by $d^\alp$, 
the algebra $\pfal(G,\ome)$ is weakly amenable if and only if $G_e$ is abelian.

{\bf (ii)}  For any $1\leq p\leq \infty$ and weight $\ome$, if $G$ is totally disconnected, then
$\pfal(G,\ome)$ is weakly amenable.
\end{proposition}

\begin{proof}
(i) If $\pfal(G,\ome)$ --- which is dense within $\fal(G)$ --- 
is weakly amenable, then so too must be $\fal(G)$.
Hence $G_e$ is abelian by \cite[Thm.\ 2.1]{forrestss}.   To see the converse we summarize
the proof of  \cite[Thm.\ 3.3]{forrestr}, with adaptations to our particular setting.
Consider a decreasing net of subgroups $(N_\iota)$
converging to $e$ for which each $G/N_\iota$ is a Lie group with open abelian connected component
of the identity.  By Proposition \ref{prop:openabel} we obtain that each 
$\pfal(G/N_\iota,\ome_{N_\iota})=\fal(G/N_\iota)$ isomorphically and, as observed above, is amenable.  
As in Lemma \ref{lem:quot}, $\pfal(G,\ome)$ is an an inductive limit of the algebras 
$T_N(\pfal(G,\ome))\cong\pfal(G/N_\iota,\ome_{N_\iota})$.   Any derivation 
$D:\pfal(G,\ome)\to\fal(G,\ome)^*$ must
vanish on each $T_N(\pfal(G,\ome))$, whence $D=0$. 

(ii)  For totally disconnected $G$ we can choose $N_\iota$ open, so $G/N_\iota$ is finite.
\end{proof}

\section{Restriction to subgroups}\label{sec:restriction}

We do not get a general usable restriction theorem, unless the subgroup is a essentially factor
in a direct product.  This latter fact gives us enough technology to characterize operator amenability of
$\pfal(G)$ for connected  groups, or products of finite groups.
We will turn our focus to the example
of a torus on $\su(2)$, where the restriction algebra seems to be of a much different form.
In doing so we gain, to our knowledge, a new class of Banach algebras of Laurent series.

Let us first consider the general situation.  We fix a closed subgroup $H$ of $G$ and
let $R_H:\fal(G)\to\fal(H)$ denote the restriction map.  Let us consider this map
on $\trig(G)$.  In this context, its adjoint is given by 
\[
R_H^\dagger(T_\sig)_{\sig\in\what{H}}
=\left(\sum_{\sig\subset\pi|_H}\sum_{k=1}^{m(\sig,\pi)}
V_{\pi,\sig}^{(k)}T_\sig V_{\pi,\sig,k}^{(k)*}\right)_{\pi\in\what{G}}
\]
where each $V_{\pi,\sig}^{(k)}:\fH_\sig\to\fH_\pi$ is an isometry, and
$m(\sig,\pi)$ is the multiplicity of $\sig$ in $\pi|_H$.  Indeed, this is straightforward
to see if $T=\lam(s)=(\sig(s))_{\sig\in\what{H}}$, for $s$ in $H$, and follows from the weak$^\dagger$
density of $\spn\lam(G)$ in $\trig(G)^\dagger$, otherwise.
By picking a suitable basis for each
$\fH_\pi$, we may suppress explicit mention of the isometries $V_{\pi,\sig}^{(k)}$ and write
\begin{equation}\label{eq:restrictionadjoint}
R_H^\dagger(T_\sig)_{\sig\in\what{H}}
=\left((T_\sig^{m(\sig,\pi)})_{\sig\subset\pi|_H}\right)_{\pi\in\what{G}}.
\end{equation}

For a weight $\ome$ we let
\[
\pfal_{G,\ome}(H)=R_H(\pfal(G,\ome))
\]
be endowed with the quotient operator
space structure which makes $R_H:\pfal(G,\ome)\to\pfal_{G,\ome}(H)$ a complete quotient map.

\begin{proposition}\label{prop:restrictoss}
The operator space structures on $\pfal_{G,\ome}(H)^*$
is determined by the completely isometric embedding
\[
T\mapsto  
\left((T_\sig)_{\sig\subset\pi|_H}\right)_{\pi\in\what{G}}
:\pfal_{G,d^\alp}(H)^*\to\ell^\infty\text{-}\bigoplus_{\pi\in\what{G}}\frac{d_\pi^{-1/p'}}{\ome(\pi)}
\left(\ell^{p'}\text{-}\bigoplus_{\sig\subset\pi|_H}m(\sig,\pi)^{1/p'}\smat^{p'}_{d_\sig}\right).
\]
\end{proposition}

\begin{proof}
The embedding result on $\pfal_{G,\ome}(H)^*$ is immediate from (\ref{eq:restrictionadjoint})
and the weighted analogue of (\ref{eq:pfaldual}).  
\end{proof}

We have not come up with an illuminating closed-form formula for the norm on $\pfal_{G,d^\alp}(H)$,
for general $p$.  For $p=1$ we obtain 
\[
\fal_{G,\ome}(H)=\fal(H,\ome_G|_H)
\] 
completely isometrically, where $\ome_G|_H$ is the restricted weight defined in (\ref{eq:restweight}).
See \cite[Prop.\ 3.5]{lees} or \cite[Prop.\ 4.12]{ludwigst}.  We shall see that even with trivial weights,
this does not hold generally for $p$-Fourier algebras.  

Let us begin by observing the case of restriction to a central subgroup.

\begin{proposition}\label{prop:resttocentral}
Let $Z$ be a closed central subgroup of $G$.  Then $\pfal_{G,\ome}(Z)=\fal(Z,\ome_G|_Z)$, completely isometrically.
\end{proposition}

\begin{proof}
It is a consequence of Schur's lemma that for $\pi$ in $\what{G}$, 
there is a character $\chi$ in $\what{Z}$ for which $\pi|_Z=\chi(\cdot)I_\pi$.  In fact
$\chi=\frac{1}{d_\pi}\trace\circ\pi|_Z$.
Furthermore, since $R_Z(\trig(G))=\trig(Z)$, each character on $Z$ is attained thusly.
Hence Proposition \ref{prop:restrictoss} yields completely isometric embedding
\[
t\mapsto \left( t_{\frac{1}{d_\pi}\trace\circ\pi|_Z}\right)_{\pi\in\what{G}}
:\pfal_G(Z)\to\ell^\infty\text{-}\bigoplus_{\pi\in\what{G}}
\frac{d_\pi^{-1/p'}}{\ome(\pi)}\left(d_\pi^{1/p'}\smat^{p'}_1\right)\cong\ell^\infty(\what{G},1/\ome).
\]
Following through to the range of this map gives us the isometric identification $\pfal_{G,\ome}(Z)^*\cong
\ell^\infty(\what{Z},(\ome_G|_Z)^{-1})$.
\end{proof}

\subsection{Direct products}\label{ssec:directproduct}
The situation of direct product groups is very nice, in this setting.
The only weighted version we shall use in the sequel is with dimension weights which 
are easier to work with as they enjoy a certain
multiplicativity with Kroenecker products.
In a direct product group $H\times K$, we may identify $H=H\times\{e\}$ and $K=\{e\}\times K$.

\begin{theorem}\label{theo:directprod}
Let $G=(H\times K)/Z$ where $Z$ is a central subgroup of $H\times K$ which satisfies
$H\cap Z=\{e\}=Z\cap K$.  Then $\pfal_G(H)=\pfal(H)$, completely isometrically.
Furthermore, we have that $\pfal_G(H,d^\alp)=\pfal(H,d^\alp)$, completely isometrically
if $Z$ is trivial, or $K$ is abelian, and completely isomorphically if $Z$ is finite.
\end{theorem}

\begin{proof}
We first recall that $\what{H\times K}=\what{H}\times\what{K}$ via Kroenecker products.
Then $\what{G}=\{\sig\times\tau:\sig\in\what{H},\tau\in\what{K}\text{ and }Z\subseteq\ker(\sig\times\tau)\}$.
Let $\sig\times\tau\in\what{G}$.
Let $p_J:H\times K\to J$ be the projection map, for $J=H,K$, and note that each $p_J(Z)$ is central
in $J$.  Hence an application of Schur's lemma tells us that there are characters $\chi,\chi'$ on
$Z$ for which $\sig\circ p_H|_Z=\chi(\cdot)I_{d_\sig}$ and $\tau\circ p_K|_Z=\chi'(\cdot)I_{d_\tau}$, while
\[
I_{d_\sig}\otimes I_{d_\tau}=\sig\times\tau|_Z=\sig\circ p_H|_Z\otimes \tau\circ p_K|_Z
=\chi\chi'(\cdot)I_{d_\sig}\otimes I_{d_\tau}
\]
which means that $\chi'=\bar{\chi}$.  

Thus we see for any $\sig$ in $\what{H}$, that there is $\tau$ in $\what{K}$ for which
$\sig\times\tau\in\what{G}$ only if 
\begin{equation}\label{eq:resttocentral}
\frac{1}{d_\sig}\trace\circ\bar{\sig}\circ p_H|_Z=\frac{1}{d_\tau}\trace\circ\tau\circ p_K|_Z.
\end{equation}
Furthermore, $\ker(p_H|_Z)=Z\cap K=\{e\}=H\cap Z=\ker(p_K|_Z)$, so $p_H(Z)\cong Z\cong p_K(Z)$.
Hence since $R_{p_J(Z)}(\trig(J))=\trig(p_J(Z))\cong\trig(Z)$ for $J=H,K$, we have
for any $\sig$ in $\what{H}$, that there exists $\tau$ in $\what{K}$ for which (\ref{eq:resttocentral})
holds.  Thus if we let $\what{K}_{Z,\sig}=\{\tau\in\what{K}:\text{(\ref{eq:resttocentral}) holds}\}$ then
we have shown that
\[
\what{G}=\{\sig\times\tau:\sig\in\what{H}\text{ and }\tau\in\what{K}_{Z,\sig}\}.
\]
It is easily checked that the condition above is symmetric in $H$ and $K$.

Proposition \ref{prop:restrictoss} now yields completely isometric embedding
\begin{align}\label{eq:dpdualembed}
T\mapsto(T_\sig^{d_\tau})_{\sig\times\tau\in\what{G}}:\pfal_{G,d^\alp}(H)^*\to
&\ell^\infty\text{-}\bigoplus_{\sig\times\tau\in\what{G}}(d_\sig d_\tau)^{-\frac{1}{p'}-\alp}
d_\tau^{1/p'}\smat^{p'}_{d_\sig} \notag \\
=& \ell^\infty\text{-}\bigoplus_{\sig\in\what{H}}d_\sig^{-\frac{1}{p'}-\alp}
\left(\ell^\infty\text{-}\bigoplus_{\tau\in\what{K}_{Z,\sig}}d_\tau^{-\alp}\smat^{p'}_{d_\sig}\right).
\end{align}
In the event that  $\sup_{\tau\in\what{K}_{Z,\sig}}d_\tau^{-\alp}=1$ (which happens
if $\alp=0$, or if $Z$ is trivial, or if $K$ is abelian),
we see that the range of the map (\ref{eq:dpdualembed}) lands in a completely isometric copy of 
$\ell^\infty\text{-}\bigoplus_{\sig\in\what{H}}d_\sig^{-\frac{1}{p'}-\alp}\smat^{p'}_{d_\sig}\cong
\pfal(H,d^\alp)^*$.  

Otherwise, letting $\chi_\sig=\frac{1}{d_\sig}\trace\circ\bar{\sig}\circ p_H|_Z$,
which is in $\what{Z}$, a simple examination of the definition of $\what{K}_{Z,\sig}$ shows that 
\[
\sup_{\tau\in\what{K}_{Z,\sig}}d_\tau^{-\alp}
=\bigl(\inf\{d^\alp_\tau:\tau\in\what{K},\chi_\sig\subset\tau\}\bigr)^{-1}
=\frac{1}{d^\alp_K|_{p_K(Z)}(\chi_\sig)}
\]
where we have a slight abuse of notation:  since $p_K(Z)\cong Z$, it would be more logical to write
$\chi_\sig\circ p_K^{-1}|_{p_K(Z)}$, instead of $\chi_\sig$, above.
The range of the map (\ref{eq:dpdualembed}) then lands in a completely isometric copy of
\[
\ell^\infty\text{-}\bigoplus_{\sig\in\what{H}}
d_\sig^{-\frac{1}{p'}-\alp}\left[d_K^\alp|_{p_K(Z)}(\chi_\sig)\right]^{-1}\smat^{p'}_{d_\sig}.
\]
Observe that $\ome^\alp(\sig)=d_K^\alp|_{p_K(Z)}(\chi_\sig)$ defines a weight on $\what{H}$
and the space above is completely isometrically isomorphic to $\pfal(H,d^\alp\ome^\alp)^*$.
In the event that $Z$ is finite, and hence so too is $\what{Z}$, there are only finitely many values 
of $\ome^\alp(\sig)=d_K^\alp|_{p_K(Z)}(\chi_\sig)$, and $\pfal(H,d^\alp\ome^\alp)$ is 
completely isomorphic to $\pfal(H,d^\alp)$.
\end{proof}

\begin{example}
Let us consider an example where $\pfal_{G,d^\alp}(H)=\pfal(H,d^\alp)$ completely
isomorphically, but not isometrically.  We have that $\un(2)=(\Tee \times\su(2))/Z$ where
$Z=\{\pm(1,I)\}$.  Then $\what{Z}=\{1,\mathrm{sgn}\}$.    Observe that in the notation above
$\ome^\alp(\mathrm{sgn})=d^\alp_{\su(2)}|_{p_{\su(2)}(Z)}(\mathrm{sgn})=2^\alp$, 
since the standard representation 
$\pi_1$ is the lowest dimension representation of $\su(2)$ which ``sees" $\mathrm{sgn}$.  
Further, letting for $k$ in $\Zee$, $\sig_k(z)=z^k$ on $\Tee$, we see, again in the notation above, that
\[
\chi_{\sig_k}=\begin{cases} 1 &\text{if }2\,|\,k \\ \mathrm{sgn} &\text{if }2\not|\; k,\end{cases}
\text{ hence }\ome^\alp(\sig_k)
=\begin{cases} 1 &\text{if }2\,|\,k \\ 2 ^\alp &\text{if }2\not|\; k.\end{cases}
\]
Thus $\pfal_{\un(2),d^\alp}(\Tee)=\pfal(\Tee,d^\alp\ome^\alp)=\fal(\Tee)$, completely isomorphically, 
though not isometrically.
\end{example}

\begin{remark}\label{rem:directprod}
Proposition \ref{prop:restrictoss} admits an obvious analogue when we replace $\pfal(G)$ by
$\fal^2_{R^q}(G,d^\alp)$; we get a completely isometric embedding
\[
\fal^2_{R^q}(G,d^\alp)^*\hookrightarrow
\ell^\infty\text{-}\bigoplus_{\pi\in\what{G}}d^{-\frac{1}{2}-\alp}\left(\ell^2\text{-}\bigoplus_{\sig\subset\pi|_H}
m(\sig,\pi)^{1/2}\smat^2_{d_\sig,R^{q'}}\right).
\]
Thus, as in Theorem \ref{theo:directprod}, in any of the situations
that $\alp=0$, $Z$ is trivial, or $K$ is abelian, we obtain that 
\[
\fal^2_{R^q,G,d^\alp}(H)=R_H(\fal^2_{R^q}(G,d^\alp))=\fal^2_{R^q}(H,d^\alp).
\]
\end{remark}

With our restriction formula in hand, we can now characterize operator amenability of our algebras for
connected Lie groups and of infinite products  of finite groups.  Let us first obtain a brief quantitative
result on finite groups.  Any finite-dimensional amenable algebra admits a cluster point of 
a bounded approximate diagonal, which we simple call a {\it diagonal}.  It is well-known
that if the algebra is commutative, then tho diagonal is unique; see, for example, 
\cite[Prop.\ 1.1]{ghandeharihs}.

\begin{proposition}\label{prop:quotbyfin}
If $G$ is finite, then the unique diagonal $w$ for $\pfal(G)=\fal^2_{R^q}(G)$ has
\begin{align*}
\norm{w}_{\pfal(G)\hat{\otimes}\pfal(G)}&=\frac{1}{|G|}\sum_{\pi\in\what{G}}d_\pi^{2+\beta(p)}
\text{ and} \\
\norm{w}_{\fal^2_{R^q}(G)\hat{\otimes}\fal^2_{R^q}(G)}&=\frac{1}{|G|}\sum_{\pi\in\what{G}}d_\pi^{
2+\gam(q)}.
\end{align*}
\end{proposition}

\begin{proof} 
Let $\fA$ be either of $\pfal$ or $\fal^2_{R^q}$.
Let $N:\fA_\Del(G)\to
\fA(G)\hat{\otimes}\fA(G)$ be given by $Nu(s,t)=u(st^{-1})$, which is an isometry
by \cite[(1.2)]{forrestss}. 
We let $1_e$ denote the indicator
function of $\{e\}$.  Then $N1_e$ is the the unique diagonal for $\fA(G)\otimes\fA(G)$
and has norm $\norm{1_e}_{\fA_\Del(G)}$.  We have that $\what{1_e}(\pi)=\frac{1}{|G|}I_{d_\pi}$
for each $\pi$, and we note that $\norm{I_{d_\pi}}_{\smat^{r(p)}_{d_\pi}}=d_\pi^{1/r(p)}$.
Hence we use Theorems \ref{theo:pfaldel} and \ref{theo:twofaldel} to finish.
\end{proof}

\begin{theorem}\label{theo:connlie}
{\bf (i)} Let $G$ be connected  and $p>1$.  Then either of $\pfal(G)$ or $\fal^2_{R^q}(G)$ 
($1\leq q\leq\infty$) is operator amenable if and only if $G$ is abelian.

{\bf (ii)}  If $G$ is connected and non-abelian, then $\pfal(G,d^\alp)$ is not operator weakly amenable for
any $p\geq \frac{4}{3+2\alp}$, and no $\fal^2_{R^q}(G)$ is operator weakly amenable.

{\bf (iii)} Let $(G_n)_{n=1}^\infty$ be a sequence of finite groups, 
$G=\prod_{n=1}^\infty G_n$ and $p>1$.  Then either of
$\pfal(G)$ or $\fal^2_{R^q}(G)$
is operator amenable if and only if all but finitely $G_n$ are abelian.
\end{theorem}

\begin{proof} (i) 
If $G$ is abelian, then $\pfal(G)=\fal^2(G)$ is amenable by Proposition \ref{prop:amenable}.  
If $G$ is not abelian then there exists  a non-empty collection of simple, connected, compact, Lie groups 
$\{S_i\}_{i\in I}$, a connected abelian group $T$, and a closed central subgroup $D$ of
$T\times S$ where $S=\prod_{i\in I}S_i$, for  which $G\cong (T\times S)/D$. 
This is the Levi-Mal'cev Theorem;
see, for example, \cite[Thm.\ 9.24]{hofmannm}.  The third isomorphism theorem tells us that
$G=(T/T\cap D\times S/D\cap S)/(D/(T\cap D\times D\cap S))$.  Thus if we let
$H=T/T\cap D$, $K=S/D\cap S$ and $Z=D/(T\cap D\times D\cap S)$, then $G=(H\times K)/Z$
where $H\cap Z=\{e\}=Z\cap K$.  Furthermore, $K\cong\prod_{i\in I}S_i/Z_i$ where each
$Z_i$ is a central subgroup of $S_i$.  Fix $i_0$ in $I$.  Then,  $K_{i_0}=S_{i_0}/Z_{i_0}$ is a 
simple Lie group.
Also, $\pfal(K_{i_0})$ is a complete quotient of $\pfal(G)$ by Theorem \ref{theo:directprod},
while $\fal^2_{R^q}(K_{i_0})$ is one of $\fal^2_{R^q}(G)$, by Remark \ref{rem:directprod}.  However,
neither of $\pfal(K_{i_0})$ nor $\fal^2_{R^q}(K_{i_0})$ 
is operator amenable by Theorem \ref{theo:tallnoa}, and the fact 
that $K_{i_0}$ is tall (\cite[Thm.\ 3.2]{hutchinson}).

(ii)  If $p\geq \frac{4}{3+2\alp}$, then the complete quotient algebra $\pfal(K_{i_0})$, 
from the proof of (i), above,
is not operator weakly amenable, thanks to Theorem \ref{theo:liewa}.  Likewise for
$\fal^2_{R^q}(K_{i_0})$.

(iii)  If all but finitely may $G_n$ are abelian, then $\pfal(G)$ is amenable by Proposition 
\ref{prop:amenable}.  If $\pfal(G)$ is operator amenable, then $\pfal(G)\hat{\otimes}\pfal(G)$
admits a bounded approximate diagonal $(w_i)$.  Let $H_n=\prod_{k=1}^nG_k$.
Then by Theorem \ref{theo:directprod} each $(R_{H_n}\otimes R_{H_n}w_i)$
is a bounded approximate diagonal for $\pfal(H_n)$, and hence has limit point
the unique diagonal $w_n$ for $\pfal(H_n)$.  Hence we appeal to Proposition \ref{prop:quotbyfin}
to see that
\begin{align*}
\sup_i\norm{w_i}_{\pfal(G)\hat{\otimes}\pfal(G)}&
\geq \sup_i \norm{R_{H_n}\otimes R_{H_n}w_i}_{\pfal(H_n)\hat{\otimes}\pfal(H_n)} \\
&\geq \norm{w_n}_{\pfal(H_n)\hat{\otimes}\pfal(H_n)}
=\frac{1}{|H_n|}\sum_{\pi\in\what{H_n}}d_\pi^{2+\beta(p)} \\
&=\frac{1}{|G_1\times\dots\times G_n|}
\sum_{\sig_1\times\dots\times\sig_k\in\what{G_1}\times\dots\times\what{G_n}}
(d_{\sig_1}\dots d_{\sig_n})^{2+\beta(p)} \\
&=\prod_{k=1}^n\frac{1}{|G_k|}\sum_{\sig_k\in\what{G_k}}d_{\sig_k}^{2+\beta(p)}.
\end{align*}
We note that $\beta(p)>0$ proveded $p>1$.
Since each $|G_k|=\sum_{\sig_k\in\what{G_k}}d_{\sig_k}^2$,
the sequence above diverges as $n\to\infty$, unless the groups $G_n$ are ultimately abelian.

Obvious modifications show the same for $\fal^2_{R^q}(G)$.  
Again we note that $\gam(q)\geq 1>0$, for all $1\geq q\geq\infty$.  \end{proof}

If we could learn the structure of $\pfal_G(G_e)$, say, where $G_e$ is the connected component
of the identity, then we may be able to asses conditions of operator amenability more generally.
Also, the situation for general totally disconnected groups is unknown to us.

\subsection{The torus in $\su(2)$}\label{ssec:tori}
We now analyze the example of the (unique up to conjugacy) torus $\Tee$ in $\su(2)$.
We will not bother with general weights, but will consider dimension weights.
We recall again that $\what{\su}(2)=\{\pi_n:n=0,1,2,\dots\}$ with $d_{\pi_n}=n+1$, and the
character theory reveals
\[
\pi_n|_\Tee=\mathrm{diag}(\chi_n,\chi_{n-2},\dots,\chi_{2-n},\chi_{-n})
\]
where $\chi_n(z)=z^n$.  Proposition \ref{prop:restrictoss}  immediately yields the following.

\begin{proposition}\label{prop:sutworest}
The operator space structure on $\pfal_{\su(2),d^\alp}(\Tee)^*\subset\Cee^\Zee$ is given by the 
completely isometric embedding
\[
t\mapsto((t_{n-2j})_{j=0}^n)_{n=0}^\infty:
\pfal_{\su(2),d^\alp}(\Tee)^*\to\ell^\infty\text{-}\bigoplus_{n=0}^\infty(n+1)^{-\frac{1}{p'}-\alp}
\ell^{p'}_{n+1}=\fM^{\alp,p'}
\]
so
\[
\norm{t}_{\pfal_{\su(2),d^\alp}(\Tee)^*}
=\sup_{n=0,1,2,\dots}\frac{\left(\sum_{j=0}^n |t_{n-2j}|^{p'}\right)^{1/p'}}{(n+1)^{\frac{1}{p'}+\alp}}.
\]
\end{proposition}

\begin{remark}
Notice that when $\alp=0$, $\pfal_{\su(2)}(\Tee)^*$ admits the description
\[
\left\{t\in\Cee^\Zee:\sup_{n=0,1,2,\dots}\left(\frac{1}{n+1}\sum_{j=0}^n |t_{n-2j}|^{p'}\right)^{1/p'}<\infty
\right\}.
\]
Thus the space is defined in terms of Cesaro summing norm.  We suspect such spaces must
exist elsewhere in the literature.
\end{remark}

\begin{corollary}\label{cor:sutworest}
The Beurling algebra $\fal(\Tee,w^{1/p'+\alp})$, where $w^{1/p'+\alp}(n)
=(1+|n|)^{\frac{1}{p'}+\alp}$, embeds completely contractively into $\pfal_{\su(2),d^\alp}(\Tee)$.
\end{corollary}

\begin{proof}
We simply observe that for $t$ in $\pfal_{\su(2),d^\alp}(\Tee)^*$, and each $n$ in $\Zee$, we have
\[
\norm{t}_{\fal(\Tee,w^{1/p'+\alp})^*}=\sup_{n\in\Zee}\frac{|t_n|}{(1+|n|)^{\frac{1}{p'}+\alp}}
\leq \norm{t}_{\pfal_{\su(2),d^\alp}(\Tee)^*}
\]
From which it follows that the adjoint of the map $\fal(\Tee,w^{1/p'+\alp})\to
\pfal_{\su(2),d^\alp}(\Tee)$, whence the map itself, is contractive.  Moreover,
$\fal(\Tee,w^{1/p'+\alp})\cong\ell^1(\Zee,w^{1/p'+\alp})$ is a maximal operator
space.
\end{proof}

\begin{remark}\label{rem:placeofresttoT}
Techniques of the proof of the corollary show that 
$\fal_{\su(2),d^\alp}(\Tee)=\fal(\Tee,w^\alp)$, completely isomorphically.  Hence
it follows from (\ref{eq:ccinclusions})
that for $p\geq 1$, $\pfal_{\su(2),d^\alp}(\Tee)$ embeds completely contractively
into $\fal(\Tee,w^\alp)$.  In particular, with completely contractive inclusions we have that
\[
\fal(\Tee,w^{1/p'+\alp})\subseteq \pfal_{\su(2),d^\alp}(\Tee)\subseteq\fal(\Tee,w^\alp).
\]
\end{remark}

We wish to study the operator weak amenability of $\pfal_{\su(2),d^\alp}(\Tee)$.
We first require an analogue of Theorem \ref{theo:pfaldel}.  In preparation we require the following,
surely well-known, estimate.

\begin{lemma}\label{lem:Tnormest}
Let $T:\Cee^n\to\Cee^m$ be linear.  Then
\[
\norm{T}_{\fC\fB(\ell^p_n,\ell^{p'}_m)}\leq 
\begin{cases}
(nm)^{\frac{1}{2p'}-\frac{1}{2p}}\norm{T}_{\fB(\ell^2_n,\ell^2_m)} & \text{if }p\geq 2 \\
\norm{T}_{\fB(\ell^2_n,\ell^2_m)} & \text{if }1\leq p <2\end{cases} 
\]
\end{lemma}

\begin{proof}
We recall the well-known equalities,
$\norm{\id}_{\fC\fB(\ell^\infty_n,\row_n)}=\norm{\id}_{\fC\fB(\ell^\infty_n,\col_n)}=n^{1/2}$.
Since $\ell^2_n=[\col_n,\row_n]_{1/2}$, we obtain that $\norm{\id}_{\fC\fB(\ell^\infty_n,\ell^2_n)}\leq
n^{1/2}$.  (Of course, equality holds.)  Now if $p\geq 2$ we have $\ell^p_n=[\ell^\infty_n,\ell^2_n]_{2/p}$,
so 
\[
\norm{\id}_{\fC\fB(\ell^p_n,\ell^2_n)}\leq \norm{\id}_{\fC\fB(\ell^\infty_n,\ell^2_n)}^{1-\frac{2}{p}}
\norm{\id}_{\fC\fB(\ell^2_n,\ell^2_n)}^{2/p}\leq n^{\frac{1}{2}(1-\frac{2}{p})}=n^{\frac{1}{2p'}-\frac{1}{2p}}.
\]
By duality $\norm{\id}_{\fC\fB(\ell^2_m,\ell^{p'}_m)}\leq m^{\frac{1}{2p'}-\frac{1}{2p}}$.
Finally we have
\[
\norm{T}_{\fC\fB(\ell^p_n,\ell^{p'}_m)}\leq \norm{\id}_{\fC\fB(\ell^2_m,\ell^{p'}_m)}
\norm{T}_{\fC\fB(\ell^2_n,\ell^2_m)}\norm{\id}_{\fC\fB(\ell^p_n,\ell^2_n)}
\leq (nm)^{\frac{1}{2p'}-\frac{1}{2p}}\norm{T}_{\fB(\ell^2_n,\ell^2_m)}
\]
where the last inequality is facilitated by the homogeneity of the operator Hilbert space structure.

If $p<2$, then since $\norm{\id}_{\fC\fB(\ell^1_n,\ell^2_n)}=1$ by virtue of $\ell^1_n$ being
a maximal space, we apply the reasoning above to $\ell^p_n=[\ell^1_n,\ell^2_n]_{2/p'}$, to see
that $\norm{\id}_{\fC\fB(\ell^p_n,\ell^2_n)}\leq 1$, likewise $\norm{\id}_{\fC\fB(\ell^2_m,\ell^{p'}_m)}\leq 1$,
and we finish accordingly.
\end{proof}

We let $\Gamma$ be as in (\ref{eq:gammamap}) with $G=\Tee$ and then
let 
\[
\pfal_{\su(2),d^\alp,\Del}(\Tee)
=\Gam\left(\pfal_{\su(2),d^\alp}(\Tee)\hat{\otimes}\pfal_{\su(2),d^\alp}(\Tee)\right).
\]

\begin{proposition}\label{prop:complicateddual}
For $t\in\Cee^\Zee$, we have 
\[
\norm{t}_{\pfal_{\su(2),d^\alp,\Del}(\Tee)^*}
\leq\begin{cases}
\underset{n=0,1,2,\dots}{\sup}(n+1)^{-1-2\alp}
\underset{k=0,\dots,n}{\max}|t_{-n+2k}| & \text{if }p\geq 2 \\
\underset{n=0,1,2,\dots}{\sup}(n+1)^{-\frac{2}{p'}-2\alp}
\underset{k=0,\dots,n}{\max}|t_{-n+2k}| & \text{if }1\leq p <2. \end{cases}
\]
\end{proposition}

\begin{proof}
We let  $\fM^{\alp,p'}$ be as in Proposition \ref{prop:sutworest}.
Since $\pfal_{\su(2),d^\alp,\Del}(\Tee)$ is a complete quotient of the predual
$\fM^{\alp,p'}_*$ which respects the dual pairing (\ref{eq:dualpairing}), and since
\[
\fM^{\alp,p'}_*\hat{\otimes}\fM^{\alp,p'}_*\cong\ell^1\text{-}\bigoplus_{m,n=0}^\infty
[(m+1)(n+1)]^{1+\frac{1}{p'}+\alp}\ell_{m+1}^p\hat{\otimes}\ell^p_{n+1}
\]
we see that the space
\[
(\fM^{\alp,p'}_*\hat{\otimes}\fM^{\alp,p'}_*)^*\cong \ell^\infty\text{-}\bigoplus_{m,n=0}^\infty
[(m+1)(n+1)]^{-1/p'-\alp}\ell^{p'}_{m+1}\check{\otimes}\ell^{p'}_{n+1}
\cong\fM^{\alp,p'}\bar{\otimes}\fM^{\alp,p'}
\]
will contain $\Gam(\pfal_{\su(2),d^\alp,\Del}(\Tee)^*)$ completely isometrically.
In particular, $t=(t_k)_{k\in\Zee}$ in $\Cee^\Zee$, we have
$\norm{t}_{\pfal_{\su(2),d^\alp,\Del}(\Tee)^*}=
\norm{\Gam^*(t)}_{\fM^{\alp,p'}\bar{\otimes}\fM^{\alp,p'}}<\infty$.

Since $\lam(s)=(s^k)_{k\in\Zee}$, for $s$ in $\Tee$, (\ref{eq:gamadj}) gives us
for $t$ in $\pfal_{\su(2),d^\alp,\Del}(\Tee)^*$ that
\[
\Gamma^*(t)=(t_kE_k\otimes E_{-k})_{k\in\Zee}\subset\Cee^{\Zee\times\Zee}
\]
where $(E_k)_{k\in\Zee}$ is the standard ``basis" in $\Cee^\Zee$.  The isometric embedding
of Proposition \ref{prop:sutworest} gives $E_k\mapsto\left(\sum_{j=0}^n\del_{k,n-2j}e_j\right)_{n=0}^\infty$,
where $(e_0,\dots,e_n)$ is the standard basis for $\ell^{p'}_{n+1}$.
Hence we find for $t$ in $\pfal_{\su(2),d^\alp,\Del}(\Tee)^*$ that
\begin{align}\label{eq:pfalsudelembed}
\Gamma^*(t)\mapsto &\left(t_k \Big(\sum_{j=0}^m\del_{k,n-2j}e_j\Big)\otimes
\Big(\sum_{i=0}^n\del_{-k,m-2i}e_i\Big)\right)_{n,m=0}^\infty \notag \\
&=\left(\sum_{k\in\Zee}t_k\eps_{m-k}\otimes\eps_{n+k}\right)_{m,n=0}^\infty
\in  \fM^{\alp,p'}\bar{\otimes}\fM^{\alp,p'}
\end{align}
where
\[
\eps_{n\pm k}=\begin{cases} e_{\frac{n\pm k}{2}} &\text{if } 2\,|\,n\pm k\text{ and } n\geq |k| \\
0 &\text{otherwise.}\end{cases}
\]
We know appeal to the fact that $\ell^{p'}_{m+1}\check{\otimes}\ell^{p'}_{n+1}\cong
\fC\fB(\ell^p_{m+1},\ell^{p'}_{n+1})$ and then to Lemma \ref{lem:Tnormest} to see that the quantity
$\norm{t}_{\pfal_{\su(2),d^\alp,\Del}(\Tee)^*}$ is dominated  in the case $p\geq 2$ by
\[
\sup_{\substack{ m,n=0,1,2,\dots \\ 2\,|\,m+n}}
[(m+1)(n+1)]^{(-\frac{1}{p'}-\alp)+(\frac{1}{2p'}-\frac{1}{2p})}
\max_{k=0,\dots,\min\{m,n\}}|t_{-\min\{m,n\}+2k}|
\]
and in the case $p<2$ by
\[
\sup_{\substack{ m,n=0,1,2,\dots \\ 2\,|\,m+n}}
[(m+1)(n+1)]^{-\frac{1}{p'}-\alp}
\max_{k=0,\dots,\min\{m,n\}}|t_{-\min\{m,n\}+2k}|.
\]
In either case the supremum is approximated by choices of $m=n$, giving the desired result.
\end{proof}

Our efforts in this section culminate in the following.

\begin{theorem}\label{theo:restpfalwa}
The algebra $\pfal_{\su(2),d^\alp}(\Tee)$ is operator weakly amenable if and only if 
$1\leq p<\frac{2}{1+2\alp}$.  Furthermore, it is weakly amenable for such $p$ and $\alp$.
\end{theorem}

\begin{proof}
Thanks to Propositions \ref{prop:wachar} and \ref{prop:derchar}, we need only to check when 
$D(f)=\left.\frac{d}{d\theta}f(e^{i\theta})\right|_{\theta=0}$ defines
a bounded element of $\pfal_{\su(2),d^\alp,\Del}(\Tee)^*$.  The functional
$D$ corresponds to the sequence $d=(k)_{k\in\Zee}$ in $\Cee^\Zee\cong\trig(\Tee)^\dagger$.  
If $p\geq 2$, it is immediate from Proposition
\ref{prop:complicateddual} that $\norm{d}_{\pfal_{\su(2),d^\alp,\Del}(\Tee)^*}<\infty$.
If $p<2$, then $\norm{d}_{\pfal_{\su(2),d^\alp,\Del}(\Tee)^*}<\infty$ if 
$(n+1)^{-\frac{2}{p'}-2\alp}n\leq (n+1)^{\frac{2}{p}-1-2\alp}$ is uniformly bounded in $n\geq 0$, 
i.e.\ if $\frac{2}{p}-1-2\alp\leq 0$,
which means that $p\geq \frac{2}{1+2\alp}$.  In these cases $\pfal_{\su(2),d^\alp,\Del}(\Tee)$ is not
operator weakly amenable, hence not weakly amenable.

Thanks to \cite{badecd}, the algebra $\fal(\Tee,w^{1/p'+\alp})\cong\ell^1(\Zee,w^{1/p'+\alp})$ is 
weakly amenable when $\frac{1}{p'}+\alp<\frac{1}{2}$, i.e.\ when $p<\frac{2}{1+2\alp}$.
Hence it follows from Corollary \ref{cor:sutworest} that $\pfal_{\su(2),d^\alp,\Del}(\Tee)$ 
weakly amenable in this case, thus also operator weakly amenable.
\end{proof}

Let us observe the following trivial consequence of Proposition \ref{prop:sutworest}, which
uses the proof of Theorem \ref{theo:sutwo}.

\begin{corollary}
The algebra $\pfal_{\su(2)}(\Tee)$ admits a bounded point derivation if and only if
$\alp\geq 1$.
\end{corollary}

Thus far, we have observed only situations in which  $\fal^2_{R^q}(G)$ exhibits the same amenability
behaviour for all $1\leq q\leq\infty$.  However, the next result will distinguish these from one another.

\begin{proposition}\label{prop:complicateddualtoo}
For $t\in\Cee^\Zee$, we have 
\[
\norm{t}_{\fal_{R^q,\su(2),d^\alp,\Del}^2(\Tee)^*}
=\sup_{n=0,1,2,\dots}(n+1)^{-1-2\alp}\left(\sum_{k=0}^n |t_{-n+2k}|^{r(q)'}\right)^{1/r(q)'}.
\]
\end{proposition}

\begin{proof}
Let $\fM^{\alp,2}_{q'}=\ell^\infty\text{-}\bigoplus_{n=0}^\infty(n+1)^{-\frac{1}{2}-\alp}\row^{p'}_{n+1}$,
which completely isometrically hosts $\fal_{R,\su(2),d^\alp,\Del}^2(\Tee)^*$, by an
obvious modification of  Proposition 
\ref{prop:sutworest}.  Then, as in the proof of Proposition \ref{prop:complicateddual}, we are
interested in testing the norm of $\Gam^*(t)$ in
\[
\ell^\infty\text{-}\bigoplus_{m,n=0}^\infty[(m+1)(n+1)]^{-\frac{1}{2}-\alp}
\row^{p'}_{m+1}\check{\otimes}\row^{p'}_{n+1}\cong
\fM^{\alp,2}_{q'}\bar{\otimes}\fM^{\alp,2}_{q'}.
\]
Appealing to the row version of (\ref{eq:cbintcol}) we see that
$\row^{p'}_{m+1}\check{\otimes}\row^{p'}_{n+1}\cong\smat^{r'}_{m+1,n+1}$ where $r'$ is chosen 
as above.
If $2\,|\,m+n$, then the anti-diagonal operator $\sum_{k\in\Zee}t_k\eps_{m-k}\otimes\eps_{n+k}$, with
notation as in the proof of Proposition \ref{prop:complicateddual}, admits norm
\[
\left(\sum_{k=0}^{\min\{m,n\}}|t_{-\min\{m,n\}+2k}|^{r(q)'}\right)^{1/r(q)'}.
\]
The supremum over all $m$ and $n$ is approximated by values where $m=n$.
\end{proof}

\begin{theorem}\label{theo:restpfalrwa}
The algebra $\fal_{R^q,\su(2),d^\alp}^2(\Tee)$ is operator weakly amenable if and only if 
$1\leq \min\{q,q'\}<\frac{2}{4\alp+1}$.  
\end{theorem}

\begin{remark}
Notice that the Banach algebra $\fal_{\su(2),d^\alp}^2(\Tee)$ is never weakly amenable 
for any $\alp\geq 0$.  Indeed, this would imply that $\fal^2_{R^2,\su(2)}(\Tee)$ is operator
weakly amenable, violating Theorem \ref{theo:restpfalwa}.
Alternatively, by Remark \ref{rem:placeofresttoT}, this would imply that
$\fal(\Tee,w^{1/2})$ is weakly amenable, violating a result of \cite{groenbaek}.
\end{remark}

\begin{proof}
In Proposition \ref{prop:complicateddualtoo} we gained an exact computation.  Hence
if we examine the derivation $D$ from the proof of Theorem \ref{theo:restpfalwa}, and its associated
sequence $d=(k)_{k\in\Zee}$, we find
\[
\norm{d}_{\fal_{R,\su(2),d^\alp,\Del}^2(\Tee)^*}
=\sup_{n=0,1,2,\dots}(n+1)^{-1-2\alp}\left(\sum_{k=0}^n (-n+2k)^{r(q)'}\right)^{1/r(q)'}.
\]
The estimate (\ref{eq:polysum}) shows that $\left(\sum_{k=0}^n (-n+2k)^{r'}\right)^{1/r'}$
grows as $(n+1)^{1+\frac{1}{r'}}$.
Hence $\norm{d}_{\fal_{R,\su(2),d^\alp,\Del}^2(\Tee)^*}<\infty$ if and only if $\frac{1}{r(q)'}-2\alp\leq 0$.
Since $r(q)=r(q')$, let us suppose that $1\leq q\leq 2$.
Then we have $\frac{1}{r(q)'}=\frac{2-q}{2q}\leq 2\alp$ if and only if $q\geq \frac{2}{4\alp+1}$.
This gives the desired result.
\end{proof}

We close this section by addressing operator amenability.  We recall, in passing, that
it is well-known that $\fal_{\su(2)}(\Tee)=\fal(\Tee)$ is (operator) amenable.

\begin{theorem}\label{theo:pfalsuteenoa}
If $p>1$, then $\pfal_{\su(2)}(\Tee)$ is not operator amenable.
The algebra $\fal^2_{R^q,\su(2)}(\Tee)$ is never operator amenable for $1\leq q\leq\infty$.
\end{theorem}

\begin{proof}
Let us begin with $\pfal_{\su(2)}(\Tee)$.
In the notation above, we let $\fM^{p'}=\fM^{0,p'}$ and consider the predual 
$\fL^p=\fM^{p'}_*$ which respects the dual pairing (\ref{eq:dualpairing}).
Let $\fD=\Gam^*(\pfal_{\su(2),\Del}(\Tee)^*)\subset\fM^{p'}\bar{\otimes}\fM^{p'}$.
We combine the observations  (\ref{eq:pfalsudelembed}), the fact that 
$\ell_{m+1}^{p'}\check{\otimes}\ell_{n+1}^{p'}\cong\fC\fB(\ell_{m+1}^p,\ell_{n+1}^{p'})$ and Lemma
\ref{lem:Tnormest} to see that in the $m,n$th component of $\fM^{p'}\bar{\otimes}\fM^{p'}$,
each $\Gam^*(\lam(s))$ ($s\in \Tee$) has norm bounded by
\[
[(m+1)(n+1)]^{-\frac{1}{p'}+\bigl( \frac{1}{2p'}-\frac{1}{2p}\bigr)}=[(m+1)(n+1)]^{-1/2}
\]
for $p>2$ and by $[(m+1)(n+1)]^{-\frac{1}{p'}}$ for $1<p<2$.
In other words $\Gam^*(\lam(s))\in\fD_0=\fD\cap\fC_0$ where
\[
\fC_0= c_0\text{-}\bigoplus_{m,n=0}^\infty
[(m+1)(n+1)]^{-1/p'}\ell^{p'}_{m+1}\check{\otimes}\ell^{p'}_{n+1}.
\]
Observe that ${\fC_0}^*\cong\fL^p\hat{\otimes}\fL^p$.

Let $\fK=\fD_\perp$, the pre-annihilator of $\fD$ in $\fL^p\hat{\otimes}\fL^p$.
Since $\Gam^*(\lam(G))$ is a weak* spanning set for $\fD$, we have that
$\fK=\bigcap_{s\in\Tee} \ker\Gam^*(\lam(s))$.  But then $\fK=\fD_0^\perp$ where
$\fD_0$ is the closed subspace generated by $\Gam^*(\lam(G))$ in $\fC_0$.
Collecting all of these facts together, we see that $\pfal_{\su(2),\Del}(\Tee)\cong
\fL^p\hat{\otimes}\fL^p/\fK$ is the dual of $\fD_0$. 

Now by Proposition \ref{prop:johnson}, if $\pfal_{\su(2)}(\Tee)$ were operator
amenable, then $\ideal_{\pfal_{\su(2),\Del}}(1)$ would admit a bounded approximate
identity.  However, since each evaluation functional $\Gam^*(\lam(s))$ is weak*-continuous,
this bounded would admit the indicator function of $\Tee\setminus\{1\}$ as a weak*-cluster
point, which is absurd.

Now we consider $\fal^2_{R^q,\su(2)}(\Tee)$.  Let $
\fM^2_{q'}=\ell^\infty\text{-}\bigoplus_{n=0}^\infty(n+1)^{-1/2}\row^{q'}_{n+1}$
and $\fL^2_q=(\fM^2_{q'})_*$ be the predual which respects the dual pairing (\ref{eq:dualpairing}). 
In $\fM^2_{q'}\bar{\otimes}\fM^2_{q'}$ the $(m,n)$-th component of
$\Gam^*(\lam(s))$ ($s\in \Tee$) is an  anti-diagonal matrix, thanks to (\ref{eq:pfalsudelembed}), 
with entries bounded in modulus by $1$.  Hence it  has norm bounded by
\[
[(m+1)(n+1)]^{-1/2}[\min\{m,n\}+1]^{1/r(q)'}\leq [(m+1)(n+1)]^{-\frac{1}{2}+\frac{1}{2r(q)'}}
\]
since we can identify that component with a weighted version of
$\smat^{r(q)'}_{m+1,n+1}$, by virtue of  (\ref{eq:cbintcol}).  However,
we always have that $\frac{1}{2r(q)'}=\frac{|q-2|}{4q}\leq \frac{1}{4}<\frac{1}{2}$. 

Just as above, the subspace $\fK=\bigcap_{s\in\Tee}\ker\Gam^*(\lam(s))\subset
\fL^2_q\hat{\otimes}\fL^2_q$ serves to allow $\fal^2_{R^q,\su(2)}(\Tee)$ to be viewed
as the dual of a space containing each $\Gam^*(\lam(s))$.  We conclude, as above.
\end{proof}


\section{Arens regularity and representability as an operator algebra}

We say that a (commutative) Banach algebra $\fA$ is {\it Arens regular} if for
any $\Phi,\Psi$ in $\fA^{**}$, and any bounded nets $(u_i),(v_j)$ from $\fA$
for which $\Phi=\lim_iu_i$ and $\Psi=\lim_jv_j$, we have that both of the weak* iterated limits
$\lim_i \lim_j u_iv_j$ and $\lim_j\lim_iu_iv_j$ exist in $\fA^{**}$ and coincide.  Arens regularity
passes to closed subalgebras and isomorphic quotients of $\fA$.  Arens regularity is a sufficient condition
to see that $\fA$ is isomorphic to a closed subalgebra of bounded operators on a reflexive Banach
space (see \cite{young}).  It is a consequence of the main result of \cite{daws} that
if $\fA$ is isomorphic to a closed subalgebra of bounded operator on a superreflexive Banach
space, then $\fA$ is Arens regular.

We say $\fA$ is {\it representable as an operator algebra}, if it is isomorphic to a closed subalgebra
of operators on a Hilbert space.  This property implies Arens regularity. 
The main result of \cite{blecher} tells us that $\fA$ is 
representable as an operator algebra provided that it admits an operator space
structure with respect to which the multiplication
on $\fA\otimes\fA$ extends to a completely bounded map on the Haagerup tensor product
$\fA\otimes^h\fA$.  If $\fA$ is already equipped with an operator space structure, we will
say it is {\it completely representable as an operator algebra} provided there is a complete 
isomorphism between $\fA$ and an operator algebra; equivalently, if multiplication
factors through $\fA\otimes^h\fA$.

Let us remark that for the weights $w^\alp(n)=(1+|n|)^\alp$ on $\Zee\cong\what{\Tee}$, we have that 
$\fal(\Tee,w^\alp)$ is Arens regular exactly when $\alp>0$ thanks to \cite{crawy,young1},
and is representable as a Q-algebra (a certain type of operator algebra) exactly when $\alp>1/2$,
thanks to \cite{varopoulos}.  As a maximal operator space, $\fal(\Tee,w^\alp)$
is completely representable an operator algebra exactly when
$\alp>1/2$, thanks to \cite{ghandeharilss}.  A study of Arens regularity of algebras
$\fal(G,\ome)$ is conducted in \cite{lees}, and of when $\fal(G,\ome)$ is completely
representable an operator algebra is conducted in \cite{ghandeharilss}.

\subsection{Arens regularity}  \label{ssec:arens}
Let us begin by showing that for an infinite $G$, the algebras $\pfal(G)$ are never
Arens regular.  We require a supporting result which is of independent interest and was observed in
\cite[Cor.\ 2.3]{forrestss1} in the case $p=2$.

\begin{proposition}\label{prop:zpfal}
Let $\cent\pfal(G)=\{u\in\pfal(G):u(tst^{-1})=u(s)\text{ for all }s,t\text{ in }G\}$.
Then $\cent\pfal(G)=\cent\fal(G)$ completely isometrically for each $p\geq 1$.
\end{proposition}

\begin{proof}
We note that for each $u$ in $\cent\pfal(G)$ we have $\hat{u}(\pi)=\frac{1}{d_\pi}
\left(\int_G u(s)\wbar{\chi_\pi(s)}\,ds\right)I_{d_\pi}$.  Moreover, for each scalar matrix, 
$\norm{\alp I_d}_{\smat^p_d}=d^{1/p}|\alp|$.  Hence each space $\cent\pfal(G)$ is
completely isometrically isomorphic to $\ell^1(\what{G})$.
\end{proof}

\begin{theorem}
The algebra $\pfal(G)$ is Arens regular only if $G$ is finite.
\end{theorem}

\begin{proof}
If $\pfal(G)$ is Arens regular, then so too must be the subalgebra $\cent\fal(G)$ and also
its ideal $\ideal_{\cent\fal}(e)=\{u\in\cent\fal(G):u(e)=0\}$.  It is a well-known consequence
of the Schur property for $\cent\fal(G)\cong\ell^1(\what{G})$ that it is weakly sequentially complete,
hence so too is the co-dimension one space $\ideal_{\cent\fal}(e)$.  By \cite[Prop.\ 3.7]{forrest}
or \cite[Thm.\ 1.5]{forrestkls}, $\ideal_\fal(e)$ admits a bounded approximate identity $(u_\alp)$.
Let $P:\fal(G)\to\cent\fal(G)$ be given by $Pu(s)=\int_G u(tst^{-1})\,dt$, which is a surjective contraction
with expectation property $P(uv)=P(u)v$ for $u$ in $\fal(G)$ and $v$ in $\cent\fal(G)$.  Then
$(Pu_\alp)$ is a bounded approximate identity for $\ideal_{\cent\fal}(e)$, as is straightforward to check.
Being weakly sequentially complete and admitting a bounded approximate identity, \cite[Cor.\ 2.4]{ulger}
tells us that $\ideal_{\cent\fal}(e)$ is a reflexive Banach space.  Hence $(Pu_\alp)$ admits a weak
cluster point,
which, by regularity of $\cent\fal(G)$ on its spectrum (the space of conjugacy classes of $G$ see remarks 
in \cite{alaghmandans}, for example), is necessarily the indicator function of $G\setminus\{e\}$.
This implies that $G$ is discrete, hence finite.
\end{proof}

Let us note a condition which implies Arens regularity.

\begin{theorem}\label{theo:arenreg}
Suppose the weight $\ome$ on $\what{G}$ satisfies
\[
\lim_{\pi\to\infty}\limsup_{\pi'\to\infty}\frac{\max_{\sig\subset\pi\otimes\pi'}\ome(\sig)}{\ome(\pi)\ome(\pi')}
=0\quad\text{ and }
\lim_{\pi'\to\infty}\limsup_{\pi\to\infty}\frac{\max_{\sig\subset\pi\otimes\pi'}\ome(\sig)}{\ome(\pi)\ome(\pi')}
=0.
\]
Then $\pfal(G,\ome)$ is Arens regular for any $p\geq 1$.
\end{theorem}

\begin{proof}
This is an easy modification of the proof of \cite[Thm.\ 3.16]{lees}.
Indeed it is easy to verify that
\[
\pfal(G,\ome)^{**}\cong\pfal(G,\ome)\bigoplus\left(c_0\text{-}\bigoplus_{\pi\in\what{G}}
\frac{d^{-1/p'}}{\ome(\pi)}\smat^{p'}_{d_\pi}\right)^\perp.
\]
All other aspects of the proof are similarly straightforward to modify.
\end{proof}

A function $\tau:\what{G}\to\Ree^{\geq 0}$ is called {\it subadditive} if 
$\tau(\sig)\leq\tau(\pi)+\tau(\pi')$ whenever $\sig\subset\pi\otimes\pi'$.  Let for $\alp>0$,
$\ome_\tau^\alp(\pi)=(1+\tau(\pi))^\alp$.  It is easy to verify $\ome_\tau$ is a weight.

\begin{corollary}\label{cor:subadweight}
Let $\tau:\what{G}\to\Ree^{\geq 0}$ be subadditive and satisfy $\lim_{\pi\to\infty}\tau(\pi)=\infty$.
Then $\pfal(G,\ome_\tau^\alp)$ is Arens regular for any $p\geq 1$ and $\alp>0$.
\end{corollary}

\begin{proof}
Note that of $\sig\subset\pi\otimes\pi'$ then 
\[
\ome_\tau^\alp(\sig)\leq (1+\tau(\pi)+\tau(\pi'))^\alp\leq(1+2\tau(\pi))^\alp+(1+2\tau(\pi'))^\alp
\leq 2^\alp [\ome_\tau^\alp(\pi)+\ome_\tau^\alp(\pi')].
\]
Hence
\[
\frac{\max_{\sig\subset\pi\otimes\pi'}\ome_\tau^\alp(\sig)}{\ome_\tau^\alp(\pi)\ome_\tau^\alp(\pi')}
\leq \frac{2^\alp}{\ome_\tau^\alp(\pi)}+ \frac{2^\alp}{\ome_\tau^\alp(\pi')}.
\]
Our assumption on $\tau$ assures that either iterated limit of Theorem \ref{theo:arenreg} is zero.
\end{proof}

\begin{example}{\bf (i)}
Let $\tau:\what{G}\to\Ree^{>0}$ be given by $\tau(\pi)=\log d_\pi$, which is subadditive.
The weight $\ome_\tau^\alp$, given here, is thus the weight $\ome_{w^\alp}$ of
Example \ref{ex:thetaweight}.  If $G$ is tall, then $\lim_{\pi\to\infty}\tau(\pi)=\infty$.

{\bf (ii)} For an infinite Lie group $G$, consider the polynomial weights 
$\ome_S^\alp(\pi)=(1+\tau_S(\pi))^\alp$, introduced in Section \ref{ssec:spectrum}.
Since each $S^{\otimes n}$ is finite, $\lim_{\pi\to\infty}\tau_S(\pi)=\infty$.
\end{example}

For the special unitary groups, a special feature of the dual allows us to deal with dimension
weights.

\begin{corollary}\label{cor:pfalsuar}
The algebra $\pfal(\su(n),d^\alp)$ is Arens regular for any $\alp>0$.
\end{corollary}

\begin{proof}
We parameterize $\what{\su}(n)$ by dominant weights:  each $\pi=\pi_\lam$ where
$\lam=(\lam_1,\dots,\lam_{n-1})$ with $\lam_1\geq\dots\geq\lam_{n-1}\geq 0$ in $\Zee$.
Then \cite[Cor.\ 1.2]{collinsls}, in whose notation $\lam_n=\mu_n=0$,  shows that 
\[
\frac{d_{\pi_\nu}^\alp}{d_{\pi_\lam}^\alp d_{\pi_\mu}^\alp}
\leq C^\alp\left(\frac{1}{\lam_1+1}+\frac{1}{\mu_1+1}\right)^\alp
\leq (2C)^\alp\left(\frac{1}{(\lam_1+1)^\alp}+\frac{1}{(\mu_1+1)^\alp}\right)
\]
for $\pi_\nu\subset\pi_\lam\otimes\pi_\mu$, where the constant $C$ depends only on $n$.
Either iterated limit of Theorem \ref{theo:arenreg} is zero.  (We remark that
the choice of generators corresponding to dominant weights $\{(1,0,\dots,0),(1,1,0,\dots,0),\dots,
(1,\dots,1,0)\}$ is such that $\tau_S(\pi_\lam)=\lam_1$ for $\lam\not=(0,\dots,0)$.  
See \cite[(3.8)]{ghandeharilss}.  Hence the present result may also be viewed as following from Corollary
\ref{cor:subadweight}.)
\end{proof}

We say a connected group $G$ is {\it semisimple} if the commutator subgroup
$[G,G]$ equals $G$.  For non-semisimple connected groups, we never have Arens regularity
with dimension weights.  This generalizes \cite[Thm.\ 4.8]{ghandeharilss}.

\begin{theorem}\label{theo:nonarens}
Let $G$ be a non-semisimple connected group.  Then $\pfal(G,d^\alp)$ is 
never Arens regular.
\end{theorem}

\begin{proof}
Thanks to \cite[9.19 \& 9.24]{hofmannm}, $G=(T\times S)/D$ where $T$ is a non-trivial
connected abelian group, and $S=\prod_{i\in I}S_i$ is a product of simple connected Lie groups,
and $D$ is totally disconnected.
Then $N=DS/D$ is a closed normal subgroup of $G$ and the third isomorphism theorem
tells us that $G/N=((T\times S)/D)(DS/D)\cong (T\times S)/SD\cong T/T\cap D$.
Since $T\cap D$ is totally disconnected, we see that $G/N$ is an infinite abelian group.
Thus by Lemma \ref{lem:quot}, $\pfal(G/N)=\pfal(G/N,d^\alp)$ is isomorphic to a closed
subalgebra of $\pfal(G,d^\alp)$.  Proposition \ref{prop:openabel} shows that $\pfal(G/N)=\fal(G/N)
\cong\ell^1(\what{G/N})$, which is never Arens regular by \cite{young1}.

Let us observe that we have an interesting alternate proof if 
$G$ is Lie.  Indeed, in the Levi-Mal'cev decomposition, $G=(T\times S)/D$ above,
we may assume again that $T$ is a non-trivial abelian connected Lie group,
$S$ is a semisimple Lie group, and $D=T\cap S$, so $T\cap D=\{e\}=D\cap S$
and $D$ is finite.  See  \cite[Thm.\ 6.15]{hofmannm}.  But then Theorem \ref{theo:directprod}
shows that $\pfal(T,d^\alp)=\fal(T)$ is an isomorphic quotient of $\pfal(G,d^\alp)$.
As above, this quotient algebra cannot be Arens regular.
\end{proof}

\subsection{Representability as an operator algebra}\label{ssec:opalg}
Let us look at some situations in which $\pfal(G,\ome)$ is representable,
in fact completely representable, as an operator algebra.
We will focus mainly on the polynomial weights $\ome_S^\alp$, defined in Section \ref{ssec:spectrum}.
We do not know if our estimates are sharp in the ranges of $p$ and $\alp$.

We will first develop an analogue of the Littlewood multipliers of \cite{ghandeharilss}.
We let $(d_\sig)_{\sig\in\Sig}$ be an indexed collection of positive integers.  
We then define
\[
\fD^p=\ell^\infty\text{-}\bigoplus_{\sig\in\Sig}d_\sig^{-\frac{1}{p'}}\smat^{p'}_{d_\sig},\quad
\fH^p=\ell^2\text{-}\bigoplus_{\sig\in\Sig}d_\sig^{\frac{1}{2}-\frac{1}{p'}}\smat^2_{d_\sig}
\quad\text{and}\quad
\fA^p=\ell_1\text{-}\bigoplus_{\sig\in\Sig}d_\sig^{1+\frac{1}{p'}}\smat^p_{d_\sig}.
\]
We observe that $\fA^*=\fD$ with respect to the dual pairing
\[
\langle (A_\sig)_{\sig\in\Sig},(D_\sig)_{\sig\in\Sig}\rangle=\sum_{\sig\in\Sig}d_\sig\trace(A_\sig D_\sig).
\]
Furthermore, with respect to this dual pairing we have linear dual
\[
\fH^{p*}=\ell^2\text{-}\bigoplus_{\sig\in\Sig}d_\sig^{\frac{1}{2}+\frac{1}{p'}}\smat^2_{d_\sig}.
\]
This being an $\ell^2$-direct sum, we have norm $\norm{X}_{\fH^{p*}}=
\left(\sum_{\sig\in\Sig}d_\sig^{1+\frac{2}{p'}}\norm{X_\sig}^2_{\smat^2_{d_\sig}}\right)^{1/2}$.
Of course, each of $\fH^p$ and $\fH^{p*}$ are also Hilbert spaces with the map
$U(H_\sig)_{\sig\in\Sig}=(d_\sig^{-1}H_\sig)_{\sig\in\Sig}$ serving as a unitary between them.
Since $U:\fH^p_C\to\fH^{p*}_C$ is a complete isometry, we still obtain operator duality
$(\fH_C)^*\cong\fH^{p*}_R$, and the same with row and column structures interchanged.

The spaces $\fD^p$ and $\fA^p$ will have their usual operator spaces structures.  However,
we write $\fD^2_C$, $\fA^2_C$ when the component spaces $\smat^2_{d_\sig}$ have
column structure; likewise for rows.
For any $A=(A_\sig)_{\sig\in\Sig}$ and $B=(B_\sig)_{\sig\in\Sig}$ in $\prod_{\sig\in\Sig}\mat_n$,
we let $AB=(A_\sig B_\sig)_{\sig\in\Sig}$.

\begin{proposition}\label{prop:somemaps}
{\bf (i)}  The formal identities $\fH^p_C\hookrightarrow\fD^p$
and $\fH^p_R\hookrightarrow\fD^p$ are normal complete contractions.

{\bf (ii)}  Given $X$ in $\fH^{p*}$, the maps $D\mapsto DX$ and $D\mapsto XD$ from $\fD^p$ into 
$\fH^p_C$, or equivalently into $\fH^p_R$, are normal and completely bounded.

{\bf (ii')} Given $X$ in $\fH^{2*}$, the maps $D\mapsto DX$ and $D\mapsto XD$ from $\fD^2_C$ 
or from $\fD^2_R$ into either of $\fH^p_C$ or $\fH^p_R$, are normal and completely bounded.
\end{proposition}

\begin{proof}
(i) It suffices the show that $\fA^p\hookrightarrow \fH^{p*}_C$ completely contractively,
and then the desired inclusion is the adjoint map.  The roles of rows and columns may be
interchanged naturally in all of our manipulations.
We first recall that the identity $d^{1/2}\col_d\hookrightarrow \row_d$ is a complete contraction.
Thus by interpolation we find that $d^{\frac{1}{2p}}\col_d^p\hookrightarrow \col_d$ is a complete
contraction. Hence we get a complete contraction
\[
d^{1/2}\smat^p_d=d^{1/2}\col_d^p\otimes^h\row_d^p
=d^{\frac{1}{2p}}\col_d^p\otimes^hd^{\frac{1}{2p'}}\col_d^{p'}
\hookrightarrow \col_d\otimes^h\col_d=\smat^2_{d,C}.
\]
Thus we obtain a complete contraction from each summand $d_\sig^{1+\frac{1}{p'}}\smat^p_{d_\sig}$
of $\fA^p$ into the summand $d_\sig^{\frac{1}{2}+\frac{1}{p'}}\smat^p_{d_\sig}\smat^2_{d_\sig,C}$
of $\fH^{p*}$.  By the universal property of direct sums, we are done.

(ii) Given $X$ in $\fH^{p*}$, the map $D\mapsto DX:\fD^p\to\fH^p_C$ will be the adjoint of the map
$H\mapsto XH:\fH^{p*}_R\to\fA^p$, once we establish that the latter is well-defined.  To see
that this latter map is well-defined, even completely bounded, is sufficient that
$X\otimes H\mapsto XH$ extends to a completely bounded map from $\fH^{p*}_C\hat{\otimes}\fH^{p*}_R$
to $\fA^p$.  Let us first observe two facts.  First, matrix multiplication
$\smat^2_{d,C}\hat{\otimes}\smat^2_{d,R}\to\smat^1_d$ is completely contractive.
Indeed, this is akin to applying trace to the middle factor of $\col_d\hat{\otimes}\smat^1_d
\hat{\otimes}\row_d=\col_d\hat{\otimes}\col_d\hat{\otimes}\row_d\hat{\otimes}\row_d$.
Second, we see by interpolation that $d^{\frac{1}{2p'}}\col_p\hookrightarrow\col_d^p$ is a complete
contraction.  Hence, similarly as in (i), above, we see that $d^{1/p'}\smat^1_d\hookrightarrow
\smat^p_d$ is completely contractive.  Now let us proceed to our multiplication computation.  
Multiplication may be realized as factoring through the following complete contractions:
\begin{align*}
\fH^{p*}_C\hat{\otimes}\fH^{p*}_R=\smat^1(\fH^{p*})
&\to \ell^1\text{-}\bigoplus_{\sig\in\Sig}\smat^1(d^{\frac{1}{2}+\frac{1}{p'}}\smat^2_{d_\sig})
=\ell^1\text{-}\bigoplus_{\sig\in\Sig}d^{1+\frac{2}{p'}}\smat^2_{d_\sig,C}\hat{\otimes}\smat^2_{d_\sig,R} \\
&\to \ell^1\text{-}\bigoplus_{\sig\in\Sig}d^{1+\frac{2}{p'}}\smat^1_{d_\sig}
\to \ell^1\text{-}\bigoplus_{\sig\in\Sig}d^{1+\frac{1}{p'}}\smat^p_{d_\sig}=\fA^p
\end{align*}
where the map in the first line is block-diagonal compression, and the two maps in the second
line are discussed above.

The commutativity of the projective tensor product, and general symmetry of row versus column operations
allows us to switch the order of the computations above rather liberally.

(ii')  With considerations so far, it is straightforward to see, for example,
 that $d^{1/2}\smat^1_d\hookrightarrow\smat^2_{d,C}$ is completely contractive.
We require this fact exactly at the last step of the multiplication computation, above.
\end{proof}

Our special Littlewood type multipliers are the content of the next theorem.
The extended (or weak*) Haagerup tensor product $\otimes^{eh}$ is defined in \cite{blechers,effrosr1}.
It is desirable for us specifically because of the completely isometric duality formula
$(\fV\otimes^h\fW)^*\cong\fV^*\otimes^{eh}\fW^*$ for any operator spaces $\fV$ and $\fW$.

\begin{theorem}\label{theo:littlewood}
Fix $T$ in $\fH^{p*}$.  Then each of the maps given by $A\otimes B\mapsto TA\otimes B$
$A\otimes B\mapsto A\otimes TB$ extends uniquely to normal linear maps
from $\fD^p\bar{\otimes}\fD^p$ to $\fD^p\otimes^{eh}\fD^p$.  In the case $p=2$,
each of these maps extends uniquely to normal linear maps from
$\fD^2_E\bar{\otimes}\fD^2_E$ to $\fD^2_E\otimes^{eh}\fD^2_E$, where $E=C$ or $R$.
\end{theorem}

\begin{proof}
Let us first observe that $\fH^p_C\bar{\otimes}\fD^p=\fH^p_C\otimes^{eh}\fD^p$, completely isometrically.
Indeed, each space is the dual of $\fH^{p*}_R\hat{\otimes}\fA^p=\fH^{p*}_R\otimes^h\fA^p$.
Then Proposition \ref{prop:somemaps} (ii) shows that $A\otimes B\mapsto TA\otimes B$
extends uniquely to a normal linear map from $\fD^p\bar{\otimes}\fD^p$ into $\fH^p_C\bar{\otimes}\fD^p
=\fH^p_C\otimes^{eh}\fD^p$.  Then Proposition  \ref{prop:somemaps} (i) shows that the formal identity
on elementary tensors, extends uniquely to a normal complete contraction
$\fH^p_C\otimes^{eh}\fD^p\longrightarrow\fD^p\otimes^{eh}\fD^p$.  The composition of these maps
yields the desired result.  The right handed case case works similarly.

The proof in the case $p=2$ is identical.  Here we use (ii') in place of (ii), from the prior proposition;
and the complete contractivity of the map $\fA^2_E\hookrightarrow\fH^{2*}_E$, $E=R$ or $C$, 
is obvious.  At this particular step, however, we may use only rows or columns.
\end{proof}

Thus, for  elements $T$ of $\fH^{p*}$, we may think of $T\otimes I$ and $I\otimes T$
in $\prod_{\sig,\sig'\in\Sig\times\Sig}\mat_{d_\sig}
\otimes \mat_{d_{\sig'}}$ as elements which multiply $\fD^p\bar{\otimes}\fD^p$
into $\fD^p\otimes^{eh}\fD^p$.

With our new Littlewood type multipliers in hand, we are almost in position to determine some occasions
for which $\pfal(G,\ome)$ is an operator algebra.  We will work only with connected
Lie groups and polynomial weights, as defined at the end of the previous section.

As in \cite[5.6.5]{wallach}, a connected Lie group $G$ admits a system of fundamental weights
$\Lam_1,\dots,\Lam_{s(G)}$ which are proportional to the simple roots of $G$, and
$\lam_1,\dots,\lam_{z(G)}$ which are related to the characters of the connected component
of the centre of $G$.  Each $\pi$ in $\what{G}$ corresponds to some $\Lam_\pi
=\sum_{j=1}^{s(G)}m_j\Lam_j+\sum_{i=1}^{z(G)}n_i\lam_i$ where each $m_j$ is a non-negative
integer and each $n_i$ is an integer.  In this notation we let
\[
\norm{\pi}_p=\left(\sum_{j=1}^{s(G)}m_j^p+\sum_{i=1}^{z(G)}|n_i|^p\right)^{1/p}.
\]
Clearly $\norm{\pi}_\infty\leq\norm{\pi}_2\leq\norm{\pi}_1\leq(s(G)+z(G))\norm{\pi}_\infty$.
The following estimate is a refinement of \cite[5.6.7]{wallach}.  For $r>2$, it will allow
us better estimates by using more refined data about $G$

\begin{lemma}\label{lem:wallach}
We have for $G$ as above and any positive real number $r$, that 
\[
\sum_{\pi\in\what{G}}\frac{d_\pi^r}{(1+\norm{\pi}_1)^{2\alp}}\quad\text{converges if}\quad
\alp>\frac{r}{4}d(G)-\frac{r-2}{4}(s(G)+z(G))
\]
where $d(G)$ is the dimension of $G$.
\end{lemma}

\begin{proof}
We simply adapt the proof of \cite[5.6.7]{wallach}.  There, it is first established that
for some constant $C$ we have $d_\pi\leq C\norm{\pi}_2^{\frac{1}{2}(d(G)-(s(G)+z(G)))}$,
where of course we can replace $\norm{\cdot}_2$ with $\norm{\cdot}_1$ at the cost of a new constant.

Now we let $\what{G}_j=\{\pi\in\what{G}:\norm{\pi}_\infty=j\}$, for which the cardinality
is estimated by 
\[
|\what{G}_j|\leq (s(G)+2z(G))(2j+1)^{s(G)+z(G)-1}\leq K(j+1)^{s(G)+z(G)-1}  
\]
for some constant $K$.  Moreover, for
$\pi$ in $\what{G}_j$ we use the estimate of the prior paragraph to see that
\[
\frac{d_\pi^r}{(1+\norm{\pi}_1)^{2\alp}}\leq\frac{C}{(1+j)^{2\alp-\frac{r}{2}d(G)+\frac{r}{2}(s(G)+z(G))}}.
\]
Thus we see that 
\begin{align*}
\sum_{\pi\in\what{G}}&\frac{d_\pi^r}{(1+\norm{\pi}_1)^{2\alp}}
= \sum_{j=0}^\infty \sum_{\pi\in\what{G}_j}\frac{d_\pi^r}{(1+\norm{\pi}_1)^{2\alp}} \\
&\leq \sum_{j=0}^\infty K(j+1)^{s(G)+z(G)-1}\frac{C}{(1+j)^{2\alp-\frac{r}{2}d(G)+\frac{r}{2}(s(G)+z(G))}}
\end{align*}
which converges provided $2\alp>\frac{r}{2}d(G)-\bigl(\frac{r}{2}-1\bigr)(s(G)+z(G))$.
\end{proof}

Let us recall how to relate the quantity $\norm{\cdot}_1$ to our polynomial weight, introduced in Section
\ref{ssec:spectrum}.   As observed in \cite[p.\ 483 \& Thm.\ 5.4]{ludwigst}, there are constants $c,C$
for which 
\begin{equation}\label{eq:polweight}
c\tau_S(\pi)\leq\norm{\pi}_1\leq C\tau_S(\pi).
\end{equation}  
The following generalizes \cite[Thm.\ 4.5]{ghandeharilss}, which deals only with the case $p=1$.

\begin{theorem}\label{theo:opalg}
Let $G$ be a connected Lie group  with polynomial weight $\ome^\alp_S$.
If 
\[
\alp>\bigl(\frac{1}{2}+\frac{1}{2p'}\bigr)d(G)-\frac{1}{2p'}(s(G)+z(G)) 
\]
then $\pfal(G,\ome^\alp_S)$ is completely representable as an operator algebra.
\end{theorem}

\begin{proof}
We will summarize those details of the proof of \cite[Thm.\ 4.5]{ghandeharilss} which need
checking.  We recall that with the coproduct $M$, and $W=(\ome(\pi)I_{d_\pi})_{\pi\in\what{G}}$,
we have that $T=M(W)(W^{-1}\otimes W^{-1})$ in $\pfal(G)^*\bar{\otimes}\pfal(G)^*$ satisfies
\[
T_{\pi,\pi'}=\left(
\left(\frac{1+\tau_S(\sig)}{(1+\tau_S(\pi))(1+\tau_S(\pi'))}\right)^\alp I_{d_\sig}
\right)_{\sig\subset\pi\otimes\pi'}
\]
in the sense of the notation (\ref{eq:pfaltpdual}).  Further, just as in the proof of Corollary 
\ref{cor:subadweight},  we gain an estimate
\[
\left(\frac{1+\tau_S(\sig)}{(1+\tau_S(\pi))(1+\tau_S(\pi'))}\right)^\alp
\leq 2^\alp\left(\frac{1}{(1+\tau_S(\pi))^\alp}+\frac{1}{(1+\tau_S(\pi'))^\alp}\right).
\]
We then let $T_1$ and $T_2$ be given by
\[
T_1=\left(\frac{1}{(1+\tau_S(\pi))^\alp}I_{d_\pi}\right)_{\pi\in\what{G}}
\quad\text{and}\quad
T_2=\left(\frac{1}{(1+\tau_S(\pi))^\alp} I_{d_{\pi'}}\right)_{\pi'\in\what{G}}
\]
Then in $\prod_{\pi,\pi'\in\what{G}\times\what{G}}\mat_{d_\pi}\otimes\mat_{d_{\pi'}}$ we can write
\[
T=S(T_1\otimes I+I\otimes T_2)
\]
where $S$ is diagonal on each block $\mat_{d_\pi}\otimes\mat_{d_{\pi'}}$ with scalars
bounded by $2^\alp$, i.e.\ $S\in\ell^\infty\text{-}\bigoplus_{\pi,\pi'\in\what{G}\times\what{G}}
\smat^\infty_{d_\pi}\check{\otimes}\smat^\infty_{d_{\pi'}}$.
It is easy to check that such $S$ acts as a multiplier of $\pfal(G)^*\bar{\otimes}\pfal(G)^*$.

Letting  $\fH^{p*}$ denote the multiplier space from Theorem \ref{theo:littlewood},
 it now suffices to find conditions for which $\norm{T_j}_{\fH^{p*}}<\infty$ for $j=1,2$.
Now we have for  $j=1,2$ that
\[
\norm{T_j}_{\fH^{p*}}^2\leq (1+C)^{2\alp}\sum_{\pi\in\what{G}}d_\pi^{1+\frac{2}{p'}}
\frac{d_\pi}{(1+\norm{\pi}_1)^{2\alp}}.
\]
thanks to (\ref{eq:polweight}).  We then appeal to Lemma \ref{lem:wallach} with the value
$r=2+\frac{2}{p'}$ to find values $\alp$ for which the last series converges.

We have seen that $M(W)(W^{-1}\otimes W^{-1})(\pfal(G)^*\bar{\otimes}\pfal(G)^*)
\subset \pfal(G)^*\otimes^{eh}\pfal(G)^*$.  By duality, this is sufficient
to see that multiplication on $\pfal(G,\ome^\alp_S)$ factors completely boundedly through the
Haagerup tensor product.  Then we appeal to the main result of \cite{blecher}.
\end{proof}

\begin{corollary}
Let $G$ be a connected Lie group.  If $\alp>\frac{3}{4}d(G)-\frac{1}{4}(s(G)+z(G))$,
then for any $1\leq q\leq\infty$ we have that
$\fal^2_{R^q}(G,\ome_S^\alp)$ is completely representable as an operator algebra.
\end{corollary}

\begin{proof}
For $q=1,\infty$, the proof of Theorem \ref{theo:opalg} can be used
to establish that for $\alp$ as above, $\fal^2_E(G,\ome^\alp_S)$ is an operator
algebra for $E=R=R^\infty$ or $C=R^1$.  We use a similar formula to (\ref{eq:interpolatedwpfal}) 
to obtain that
\[
\fal^2_{R^q}(G,\ome^\alp_S)=\bigl[\fal^2_R(G,\ome^\alp_S),\fal^2_C(G,\ome^\alp_S)\bigr]_{1/q}.
\]
By \cite[2.3.7]{blecherl}, this interpolated algebra
is completely isomorphic to an operator algebra.
\end{proof}

Let us close by returning to dimension weights, although only for special unitary groups.

\begin{corollary}\label{cor:pfalsuopalg}
If $\alp>\bigl(\frac{1}{2}+\frac{1}{2p'}\bigr)(n^2-1)-\frac{1}{2p'}(n-1)$, then $\pfal(\su(n),d^\alp)$
is completely representable as an operator algebra.
\end{corollary}

\begin{proof}
We use notation and comments in the proof of Corollary \ref{cor:pfalsuar}.
We have that $\lam_1=\norm{\pi_\lam}_1$ in this particular case.  The main
estimate of that proof hence carries over, whence so too does the proof of
Theorem \ref{theo:opalg} and the last corollary.
\end{proof}

We remark that $\fal(\su(n),d^\alp)$ is known not to be completely isomorphic to an operator algebra
when $\alp\leq 1/2$ (\cite[4.11]{ghandeharilss}).  With the difficulty of the restriction result to
tori, even for $n=2$, we have no such result for $p\not=1$, at present.

\section{Summary of results}

\subsection{The tables.} We summarize our results in tables, below.  We require labels for our first table.

\smallskip
\begin{tabular}{ll}
A = amenable & WA = weakly amenable \\
OA = operator amenable & OWA = operator weakly amenable \\
\multicolumn{2}{l}{PD = admits bounded point derivation} \\ 
v.a. = virtually abelian & c.n.a. = connected non-abelian
\end{tabular}

\smallskip\noindent
In each table, known ranges of the validity of that property
will be given.  If a quantitative range is, moreover, known to us to be sharp we shall indicate so
with an (s).  All results below can be found in the present article, in \cite{lees}, or references therein.

\smallskip
\begin{tabular}{|l|c|c|c|c|c|}
\hline
& A & WA & OA & OWA & PD \\ \hline 
$\pfal(G,d^\alp)$ & always & always & always  & always & never \\
$G$ v.a. & & & & & \\ \hline
$\pfal(G,d^\alp)$ & $G$ v.a. (s)& always & $p=1$ \& & always  & never \\
$G_e$ abelian &  & 		& $\alp = 0^{(\ast)}$ & & \\ \hline
$\pfal(\su(2),d^\alp)$ & never & never & $p=1$ \& & $p<\frac{4}{3+2\alp}$ (s) & 
$\alp\geq 1$ (s) \\ 
& & & $\alp=0$ (s) & & \\ \hline
$\pfal(G,d^\alp)$ & never & never & $p=1$ \& & no if  & 
$\alp\geq 1$  \\
$G$ c.n.a.  & & & $\alp=0$ (s) & $p\geq\frac{4}{3+2\alp}$& \\ \hline
$\fal^2_{R^q}(G,d^\alp)$ & never & never & never & never & $\alp\geq 1$ \\
$G$ c.n.a. & & & & & \\ \hline
$\pfal_{\su(2),d^\alp}(\Tee)$ & $p=1$ \& &
$p<\frac{2}{1+2\alp}$ (s)  & $p=1$ \& & $p<\frac{2}{1+2\alp}$ (s) & $\alp\geq 1$ (s) \\ 
&  $\alp=0$ (s) &  &$\alp=0$ (s)  &  &\\ \hline
$\fal^2_{R^q,\su(2),d^\alp}(\Tee)$ & never & never & never & $q<\frac{2}{1+4\alp}$ or & $\alp\geq 1$ (s) \\
& & & &  $q'<\frac{2}{1+4\alp}$ (s) & \\
\hline
\end{tabular}

\noindent ($*$) Sharp for tall groups and those groups with direct product 
factors which are infinite products of non-abelian
finite groups.

\medskip
\begin{tabular}{|l|c|c|}
\hline
& Arens regular & completely representable  \\ 
& & as an operator algebra \\ \hline
$\pfal(G)$ & never & never \\ \hline
$\pfal(G,d^\alp)$, connected & never & never \\
non-semisimple  & & \\ \hline
$\pfal(G,\ome^\alp_S)$, infinite & $\alp>0$ (s) & 
$\alp>\bigl(\frac{1}{2}+\frac{1}{2p'}\bigr)d(G)-\frac{1}{2p'}(s(G)+z(G))$ \\
connected Lie & & \\ \hline
$\fal^2_{R^q}(G,\ome^\alp_S)$, infinite & $\alp>0$ (s) & 
$\alp>\frac{3}{4}d(G)-\frac{1}{4}(s(G)+z(G))$ \\
connected Lie & & \\ \hline
$\pfal(\su(n),d^\alp)$ & $\alp>0$ (s) & 
$\alp>\bigl(\frac{1}{2}+\frac{1}{2p'}\bigr)(n^2-1)-\frac{1}{2p'}(n-1)$ \\ \hline
\end{tabular}

\subsection{Questions.}
The following questions were partially addressed, and arose naturally in the course of this
investigation.

{\bf (a)}  If $G$ is a disconnected Lie group and $\ome_S^\alp$ a polynomial weight, 
is $\fal_{G,\ome_S^\alp}(G_e)$ always regular?  [If this is shown to be true, we can conclude that
$\fal(G,\ome_S^\alp)$ is always regular, and then extend this to $\pfal(G,\ome_S^\alp)$.]

{\bf (b)}  If $\fal^2(G)\cong\fal^2(H)$ isometrically, must we have $G\cong H$, topologically?

{\bf (c)}  Under what general conditions is $\pfal(G)$ operator amenable for $p>1$?

{\bf (d)}  Under what general conditions is $\pfal(G,d^\alp)$ operator weakly amenable?

{\bf (e)}  Given a connected non-abelian Lie group $G$, does $p<\frac{4}{3+2\alp}$
imply that  $\pfal(G,d^\alp)$ is operator weakly amenable?  How about the case of
 $G$ being semi-simple?

{\bf (f)} If $G$ is a tall infinite group, and $H$ is any closed subgroup which meets every
conjugacy class of $G$, is $\pfal_G(H)$ is ever operator amenable for any $p>1$.

{\bf (g)}  Is there a good general description of $\pfal_{G,d^\alp}(G_e)$?

{\bf (h)}  What are the sharp bounds for complete representability as an operator algebra
of $\fal^p(G,\ome^\alp_S)$ for a connected Lie group $G$?  How about for $G=\su(n)$?  
How about sharp bounds for representability as an operator algebra?

Addresses:
\linebreak
 {\sc 
Department of Mathematical Sciences, Seoul National University,
San56-1 Shinrim-dong Kwanak-gu, Seoul 151-747, Republic of Korea \\
Department of Mathematics and Statistics, University of Saskatchewan,
Room 142 McLean Hall, 106 Wiggins Road
Saskatoon, SK, S7N 5E6  \\
Department of Pure Mathematics, University of Waterloo,
Waterloo, ON, N2L 3G1, Canada.}

\medskip
Email-adresses:
\linebreak
{\tt hunheelee@snu.ac.kr}
\linebreak {\tt samei@math.usask.ca}
\linebreak {\tt nspronk@uwaterloo.ca}

\end{document}